\documentclass{amsart}

\usepackage{url}
\usepackage{amsfonts}
\usepackage{amsmath}
\usepackage{amssymb}
\usepackage{amsthm}
\usepackage{stmaryrd}
\usepackage{amscd}
\usepackage{mathrsfs}
\usepackage{textcomp}
\usepackage{mathtext}
\usepackage{yfonts}

\usepackage{enumerate}

\usepackage{tikz}

\newtheorem{theorem}{Theorem}[section]
\newtheorem{lemma}[theorem]{Lemma}
\newtheorem{corollary}[theorem]{Corollary}
\theoremstyle{definition}

\newtheorem{example}[theorem]{Example}
\theoremstyle{remark}
\newtheorem{remark}[theorem]{Remark}

\numberwithin{equation}{section}

\usepackage{nicefrac}
\newcommand{\niceonehalf}{{\nicefrac{1}{2}}}

\newcommand{\vol}{\operatorname{vol}}
\newcommand{\eps}{{\epsilon}}

\newcommand{\inv}{{-1}}

\newcommand{\Jacobian}{\operatorname{D}}
\newcommand{\dif}{{\mathrm d}}
\newcommand{\grad}{\operatorname{grad}}
\newcommand{\curl}{\operatorname{curl}}

\newcommand{\divergence}{\operatorname{div}}

\newcommand{\equivalent}{ \Longleftrightarrow }

\newcommand{\sgn}{\operatorname{sgn}}
\newcommand{\st}{ \mid }

\newcommand{\diam}{{\operatorname{diam}}}

\newcommand{\signum}{\operatorname{sgn}}

\newcommand{\Lip}{\operatorname{Lip}}
\newcommand{\trace}{\operatorname{tr}}
\newcommand{\Id}{\operatorname{Id}}

\newcommand{\convex}{\operatorname{convex}}

\newcommand{\cartan}{{\mathsf d}}
\newcommand{\cartanx}{{{\mathsf d}x}}

\newcommand{\bbN}{{\mathbb N}}

\newcommand{\bbR}{{\mathbb R}}

\newcommand{\bbZ}{{\mathbb Z}}

\newcommand{\calA}{{\mathcal A}}

\newcommand{\calC}{{\mathcal C}}

\newcommand{\calP}{{\mathcal P}}

\newcommand{\calT}{{\mathcal T}}

\newcommand{\ttH}{{\mathtt H}}

\newcommand{\scrR}{{\mathscr R}}

\newcommand{\vecX}{{\vec X}}

\begin{document}

\title[Smoothed Projections Weakly Lipschitz]{Smoothed Projections over\\ Weakly Lipschitz Domains}

\author[Martin Werner Licht]{Martin Werner Licht}
\address{Department of Mathematics, University of Oslo, P.O. Box 1053 Blindern, NO-0316 Oslo, Norway}
\email{martinwl@math.uio.no}
\thanks{This research was supported by the European Research Council through 
the FP7-IDEAS-ERC Starting Grant scheme, project 278011 STUCCOFIELDS.}


\date{}

% \dedicatory{}

\begin{abstract}
 We develop finite element exterior calculus over weakly Lipschitz domains.
 Specifically, we construct commuting projections from $L^p$ de~Rham complexes 
 over weakly Lipschitz domains onto finite element de~Rham complexes.
 These projections satisfy uniform bounds for finite element spaces 
 with bounded polynomial degree over shape-regular families of triangulations.
 Thus we extend the theory of finite element differential forms
 to polyhedral domains that are weakly Lipschitz but not strongly Lipschitz.
 As new mathematical tools,
 we use the collar theorem in the Lipschitz category,
 and we show that the degrees of freedom in finite element exterior calculus
 are flat chains in the sense of geometric measure theory.
\end{abstract}

\maketitle

% I. Introduction 

\section{Introduction}
\label{sec:introduction}

The aim of this article is to contribute to the understanding of finite element methods
for partial differential equations over domains of low regularity. 
For partial differential equations associated to a differential complex,
projections that commute with the relevant differential operators 
are central to the analysis of mixed finite element methods.
In particular, \emph{smoothed projections} from Sobolev de~Rham complexes 
to finite element de~Rham complexes
are used in \emph{finite element exterior calculus} (FEEC) \cite{AFW1,AFW2}.
This was researched when the underlying domain is a Lipschitz domain.
In this article, we regard more generally finite element exterior calculus
when the underlying domain is merely a \emph{weakly Lipschitz domain}.
Specifically, we construct and analyze smoothed projections. 
% Our principle achievement is providing a detailed construction
% and analysis of smoothed projections when the underlying domain is a \emph{weakly Lipschitz domain}. 
% This generalizes previous results
Thus we enable the abstract Galerkin theory of finite element exterior calculus
within that generalized geometric setting.

It is easy to provide motivation for considering the class of weakly Lipschitz domains
in the context of finite element methods.
A domain is called weakly Lipschitz if its boundary can be flattened locally
by a Lipschitz coordinate transformation.
% Let us motivate our interest in this class of domains.
This generalizes the classical notion of (strongly) Lipschitz domains,
whose boundaries, by definition, can be written locally as Lipschitz graphs.
Although Lipschitz domains are a common choice for the geometric ambient
in the theoretical and numerical analysis of partial differential equations,
they exclude several practically relevant domains.
% that appear in finite element theory.
It is easy to find three-dimensional polyhedral domains that are not Lipschitz domains,
such as the ``crossed bricks domain'' \cite[p.39, Figure 3.1]{Monk2003}. %FIXME: Check the page 
But as we show in this article, every three-dimensional polyhedral domain
is still a weakly Lipschitz domain. 
% So this class of domains appears in applications.

% This application renders weakly Lipschitz domains interesting for finite element analysis.
% FIXME: How long has there been interest? Find a good replacement for 'Moreover'.
Moreover, weakly Lipschitz domains have attracted interest 
in the theory of partial differential equations 
because basic results in vector calculus, well-known for strongly Lipschitz domains,
are still available in this geometric setting 
\cite{grisvard1980boundary, picard1984elementary, hofmann2007geometric, GMM, brewster2013weighted, bauer2015maxwell}.
For example, one can show that the differential complex
\begin{align}
 \begin{CD}
  H^{1}(\Omega) 
  @>\grad>>
  H(\curl,\Omega)
  @>\curl>>
  H(\divergence,\Omega)
  @>\divergence>>
  L^{2}(\Omega)
 \end{CD}
\end{align} 
over a bounded three-dimensional weakly Lipschitz domain $\Omega$
satisfies Poincar\'e-Friedrichs inequalities, 
and realizes the Betti numbers of the domain on cohomology. 
Furthermore, a vector field version of a Rellich-type compact embedding theorem 
is valid, and the scalar and vector Laplacians over $\Omega$ have a discrete spectrum.
Recasting this in the calculus of differential forms,
one can more generally establish the analogous properties for the $L^{2}$ de~Rham complex
\begin{align}
 \label{intro:derham:L2}
 \begin{CD}
  H\Lambda^{0}(\Omega) 
  @>\cartan>>
  H\Lambda^{1}(\Omega) 
  @>\cartan>>
  \cdots
%   @>\cartan>>
%   H\Lambda^{n-1}(\Omega) 
  @>\cartan>>
  H\Lambda^{n}(\Omega) 
 \end{CD}
\end{align}
over a bounded weakly Lipschitz domain $\Omega \subset \bbR^{n}$.

%%%%%% Motivation for FEM
%%%%%% 
%%%%%% 

% We are interested to develop finite element analysis over weakly 
It is therefore of interest to develop finite element analysis over weakly Lipschitz domains. 
% We regard these results as a motivation to develop finite element analysis over weakly Lipschitz domains.
Since the analytical theory is formulated within the calculus of differential forms,
we wish to adopt this calculus on the discrete level.
Specifically, we use the framework of finite element exterior calculus,
and our agenda is to extend that framework to numerical analysis on weakly Lipschitz domains.
% to our generalized geometric setting.
% In order to replicate the analytical setting on a discrete level,
% we adopt the framework of finite element exterior calculus. 
The foundational idea is to study a finite element de~Rham complex
\begin{align}
 \label{intro:derham:discrete}
 \begin{CD}
  \Lambda^{0}(\calT)  @>\cartan>> \Lambda^{1}(\calT)  @>\cartan>> \cdots @>\cartan>> \Lambda^{n}(\calT)
 \end{CD}
\end{align}
that mimics the $L^{2}$ de~Rham complex. 
Here, each $\Lambda^{k}(\calT)$ is a subspace of $H\Lambda^{k}(\Omega)$
whose members are piecewise polynomial with respect 
to a fixed triangulation $\calT$ of the domain. 
Arnold, Falk, and Winther \cite{AFW1} have classified 

A central component of finite element exterior calculus are uniformly bounded \emph{smoothed projections}.
Our main contribution (Theorem~\ref{prop:finalprojection}) in this article is to devise such a projection
when the domain is merely weakly Lipschitz. 
A condensed version of our result is the following theorem.

\begin{theorem}
 Let $\Omega \subseteq \bbR^{n}$ be a bounded weakly Lipschitz domain,
 and let $\calT$ be a simplicial triangulation of $\Omega$.
 Let \eqref{intro:derham:discrete} be a differential complex of finite element spaces of differential forms
 as in finite element exterior calculus (\cite{AFW1}).
 Then there exist bounded linear projections $\pi^{k} : L^{2}\Lambda^{k}(\Omega) \rightarrow \Lambda^{k}(\calT)$
 such that the following diagram commutes:
 \begin{align}
  \begin{CD}
   H\Lambda^{0}(\Omega) @>\cartan>> H\Lambda^{1}(\Omega) @>\cartan>> \cdots @>\cartan>> H\Lambda^{n}(\Omega)
   \\
   @V{\pi^{0}}VV @V{\pi^{1}}VV @. @V{\pi^{n}}VV
   \\
   \Lambda^{0}(\calT)  @>\cartan>> \Lambda^{1}(\calT)  @>\cartan>> \cdots @>\cartan>> \Lambda^{n}(\calT).
  \end{CD}
 \end{align}
 Moreover, $\pi^{k} \omega = \omega$ for $\omega \in \Lambda^{k}(\calT)$.
 The operator norm of $\pi^{k}$ is bounded uniformly in terms of 
 the maximum polynomial degree of \eqref{intro:derham:discrete},
 the shape measure of the triangulation,
 and geometric properties of $\Omega$.
\end{theorem}

As an immediate consequence, the a priori error estimates of finite element exterior calculus 
are applicable over weakly Lipschitz domains. Commuting projection operators 
have been approached from different perspectives in the theory of finite elements
\cite{christiansen2007stability,schoberl2008posteriori,gopalakrishnan2012commuting,demlow2014posteriori}.
\\

Let us outline the construction of the smoothed projection
and the new tools which we employ.
We largely follow ideas in published literature \cite{AFW1,christiansen2008smoothed}
but introduce significant technical modifications in this article.
Given a differential form over the domain,
the smoothed projection is constructed in several steps.
% Our major point of generalization consists in including the class of weakly Lipschitz domains.

We first extend the differential form beyond the original domain by reflection along the boundary,
using a parameterized tubular neighborhood of the boundary.
For strongly Lipschitz domains, such a parametrization can be constructed
using the flow along a smooth vector field transversal to the boundary \cite{AFW1,christiansen2008smoothed},
but for weakly Lipschitz domains such a transversal vector field does not necessarily exist.
Instead we obtain the desired parameterized tubular neighborhood
via a variant of the collaring theorem in Lipschitz topology \cite{luukkainen1977elements}.

Next, a mollification operator smoothes the extended differential form.
In order to guarantee uniform bounds for shape-regular families of meshes, 
the mollification radius is locally controlled by a smoothed mesh size function.
This is similar to \cite{christiansen2008smoothed},
but we elaborate the details of the construction and make a minor correction;
see also Remark~\ref{rem:mollifier}.
We find that the mollified differential form has well-defined degrees of freedom.

We then apply the canonical finite element interpolator to the mollified differential form. 
The resulting \emph{smoothed interpolator} commutes with the exterior derivative
and satisfies uniform bounds, but it is generally not idempotent. 
We can, however, control the interpolation error over the finite element space.
If the smoothed interpolator is sufficiently close to the identity over the finite element space,
then a commuting and uniformly bounded discrete inverse exists.
Following an idea of Sch\"oberl \cite{schoberl2005multilevel}, 
the composition of this discrete inverse with the smoothed interpolator
yields the desired smoothed projection. 

In order to derive the aforementioned interpolation error estimate over the finite element space,
we call on geometric measure theory \cite{federer2014geometric,whitney2012geometric}.
The principle motivation in utilizing geometric measure theory is the low regularity of the boundary,
which requires new techniques in finite element theory.
A key observation, which we believe to be of independent interest,
is the identification of the degrees of freedom as \emph{flat chains}
in the sense of geometric measure theory.
The desired estimate of the interpolation error over the finite element space
is proven eventually with distortion estimates on flat chains.
To the author's best understanding, the results in this article
also provide non-trivial details for some proofs in the aforementioned references,
hitherto not available in literature; see also Remark~\ref{rem:interpolationerror}.

Most of literature on commuting projections focuses on the $L^{2}$ theory 
(but see also \cite{ern2016mollification}).
We consider differential forms with coefficients in general $L^{p}$ spaces, following \cite{gol1982differential}.
This article moreover prepares future research on smoothed projections which preserve partial boundary conditions.
\\

%%%%%
%%%%% ... is structured as follows.
%%%%%

The remainder of this work is structured as follows.
In Section~\ref{sec:geometry}, we introduce weakly Lipschitz domains and a collar theorem.
We recapitulate the calculus of differential forms in Section~\ref{sec:differentialforms}.
We briefly review triangulations in Section~\ref{sec:triangulations}.
The relevant background in geometric measure theory is given in Section~\ref{sec:gmt}.
Then we introduce finite element spaces, degrees of freedom, and interpolation operators 
in Section~\ref{sec:interpolator}.
In Section~\ref{sec:projection}, we finally construct the smoothed projection.

% II. Geometric setting

\section{Geometric Setting}
\label{sec:geometry}

We begin by establishing the geometric background. 
We review the notion of \emph{weakly Lipschitz domains}
and prove the existence of a \emph{closed two-sided Lipschitz collar} 
along the boundaries of such domains.
We refer to \cite{luukkainen1977elements} for further background in Lipschitz topology. 
\\

Throughout this article, and unless stated otherwise, 
we let finite-dimensional real vector spaces $\bbR^{n}$
and their subsets be equipped with the canonical Euclidean metrics.
We let $B_r(U)$ be the closed Euclidean $r$-neighborhood, $r > 0$, 
of any set $U \subseteq \bbR^n$,
and we write $B_r(x) := B_r(\{x\})$.
% for the Euclidean ball of radius $r > 0$
% centered at $x \in \bbR^n$.

We introduce some basic notions of Lipschitz analysis.
Let $X \subseteq \bbR^{n}$ and $Y \subseteq \bbR^{m}$,
and let $f : X \rightarrow Y$ be a mapping.
For a subset $U \subseteq X$, we let 
the \emph{Lipschitz constant} $\Lip(f,U) \in [0,\infty]$ of $f$ over $U$
be the minimal $L \in [0,\infty]$ that satisfies
\begin{align*}
 \forall x, x' \in U :
 \left\| f(x) - f(x') \right\| \leq L \| x - x' \|
%  ,
%  \quad 
%  x, x' \in U
 .
\end{align*}
% We call $\Lip(f,U)$ the \emph{Lipschitz constant} of $f$ over $U$.
% We simply write $\Lip(f)$ if $X$ is understood.
We call $f$ \emph{Lipschitz} if $\Lip(f,X) < \infty$.
% and simply write $\Lip(f)$ if $X$ is understood.
We call $f$ \emph{locally Lipschitz} or \emph{LIP}
if for each $x \in X$ there exists a relatively open neighborhood $U \subseteq X$ of $x$
% that is relatively open in $X$ and 
such that $f_{|U} : U \rightarrow Y$ is Lipschitz.
If $f$ is invertible, then we call $f$ \emph{bi-Lipschitz}
if both $f$ and $f^{\inv}$ are Lipschitz,
and we call $f$ a \emph{lipeomorphism} if both $f$ and $f^{\inv}$
are locally Lipschitz.
If $f : X \rightarrow Y$ is locally Lipschitz and injective
such that $f : X \rightarrow f(X)$ is a lipeomorphism,
then we call $f$ a \emph{LIP embedding}.
The composition of Lipschitz mappings is again Lipschitz,
and the composition of locally Lipschitz mappings is again locally Lipschitz. 
If $X$ is compact, then every locally Lipschitz mapping is also Lipschitz.
\\

Let $\Omega \subseteq \bbR^{n}$ be open.
We call $\Omega$ a \emph{weakly Lipschitz domain}
if for all $x \in \partial\Omega$ there exist a closed neighborhood
$U_x$ of $x$ in $\bbR^n$ and a bi-Lipschitz mapping
$\phi_x : U_x \rightarrow [-1,1]^n$ such that 
$\phi_x(x) = 0$ and such that 
\begin{subequations}
\label{math:weaklylipschitzdomain}
\begin{align}
 \label{math:weaklylipschitzdomain:condition1}
 \phi_x( \Omega   \cap U_x )       &= [-1,1]^{n-1} \times [-1,0)
 , \\
 \label{math:weaklylipschitzdomain:condition2}
 \phi_x( \partial\Omega \cap U_x ) &= [-1,1]^{n-1} \times \{0\}
 , \\
 \label{math:weaklylipschitzdomain:condition3}
 \phi_x( \overline\Omega^c \cap U_x )       &= [-1,1]^{n-1} \times (0,1]
 .
\end{align}
\end{subequations}
The closed sets $\{ \partial\Omega \cap U_x \st x \in \partial\Omega \}$
cover $\partial\Omega$ and the mappings 
$\phi_{x|\partial\Omega \cap U_x} : \partial\Omega \cap U_x \rightarrow [-1,1]^{n-1}$
are bi-Lipschitz. Note that $\Omega$ is a weakly Lipschitz domain
if and only if $\overline\Omega^{c}$ is a weakly Lipschitz domain.

\begin{remark}
 In other words, a weakly Lipschitz domain is a domain
 whose boundary can be flattened locally by a bi-Lipschitz 
 coordinate transformation.
 The notion of \emph{weakly Lipschitz domain} contrasts 
 with the classical notion of \emph{Lipschitz domain},
 then also called \emph{strongly Lipschitz domain}.
 A strongly Lipschitz domain is an open subset $\Omega$ of $\bbR^{n}$
 whose boundary $\partial\Omega$
 can be written locally as the graph of a Lipschitz function
 in some orthogonal coordinate system.
 Strongly Lipschitz domains are weakly Lipschitz domains,
 but the converse is generally false.
%  See also \cite{grisvard1980boundary} for a discussion.
 
 A different access towards the idea originates 
 from differential topology: a weakly Lipschitz domain
 is a \emph{locally flat Lipschitz submanifold} of $\bbR^{n}$
 in the sense of \cite{luukkainen1977elements}.
 This idea has motivated the notion of weakly Lipschitz domains
 inside general Lipschitz manifolds \cite{GMM}.
\end{remark}

\begin{example}
 Every bounded domain $\Omega \subset \bbR^{3}$ with a finite triangulation is a weakly Lipschitz domain.
 We will make this statement precise, and provide a proof, in Section~\ref{sec:triangulations}
 after having formally introduced triangulations.
 At this point, let us consider a concrete and well-known example,
 namely the crossed bricks domain, 
 which we already mentioned in the introduction. 
 Let 
 %   A well-known example of a weakly Lipschitz domain,
 %   already mentioned in the introduction, 
 %   is the crossed bricks domain. Let 
 \begin{align}\begin{split}
  \Omega_{CB} 
  := 
  (-1,1) \times (0,1) \times (0,-1) 
  % oberer klotz --- x laenge, y-z quadratisch
  &\cup
  ( 0,1) \times (0,-1) \times (-1,1) 
  % unterer klotz --- x-y quadratisch, z die laenge
  \\&\cup
  (0,1) \times \{0\} \times (0,-1)
  % kleber zwischen den beiden kloetzen
  % (x,y,z)
  .
 \end{split}\end{align}
 At the origin, $\partial\Omega_{CB}$ is not the graph of a Lipschitz function
 in any coordinate system. But whereas $\Omega_{CB}$ is not a strongly Lipschitz domain,
 it is still a weakly Lipschitz domain.
 To see this, we first observe that near every non-zero $z \in \partial\Omega_{CB}$ 
 we can write $\partial\Omega$ as a Lipschitz graph, 
 from which we can easily construct a suitable Lipschitz coordinate chart around $z$.
 Finally, to obtain a bi-Lipschitz coordinate chart at the origin, 
 we use a bi-Lipschitz mapping to transform $\Omega_{CB}$ into a domain that is a Lipschitz graph
 in a neighborhood of the origin.
 See also Figure~\ref{fig:crossedbrickspicture}.
 % 
%  For a generalization of this example, we refer to Example 2.2 in \cite{axelsson2004hodge}.
\end{example}

\begin{figure}[t]
 \begin{tabular}{cc}
  \begin{tikzpicture}
  % [line join=bevel,z={( 3.85mm, -3.85mm)}]
  [line join=bevel,x={( 1.5cm, 0mm)},y={( 0mm, 1.5cm)},z={( 1.5*3.85mm, -1.5*3.85mm)}]

  \coordinate (UL1) at ( 0, 0,-1);
  \coordinate (UL2) at ( 0,-1,-1);
  \coordinate (UL3) at ( 1, 0,-1);
  \coordinate (UL4) at ( 1,-1,-1);
  \coordinate (UR1) at ( 0, 0, 1);
  \coordinate (UR2) at ( 0,-1, 1);
  \coordinate (UR3) at ( 1, 0, 1);
  \coordinate (UR4) at ( 1,-1, 1);

  %orange!80!blue
  \draw (UL1) -- (UL2) -- (UL4) -- (UL3) -- cycle;
  \draw [fill opacity=0.7,fill=yellow] (UL1) -- (UL2) -- (UL4) -- (UL3) -- cycle;
  \draw (UL2) -- (UL4) -- (UR4) -- (UR2) -- cycle;
  \draw [fill opacity=0.7,fill=yellow] (UL2) -- (UL4) -- (UR4) -- (UR2) -- cycle;
  \draw (UL1) -- (UL2) -- (UR2) -- (UR1) -- cycle;
  \draw [fill opacity=0.7,fill=yellow] (UL1) -- (UL2) -- (UR2) -- (UR1) -- cycle;
  \draw (UL3) -- (UL4) -- (UR4) -- (UR3) -- cycle;
  \draw [fill opacity=0.7,fill=yellow] (UL3) -- (UL4) -- (UR4) -- (UR3) -- cycle;
  \draw (UL1) -- (UL3) -- (UR3) -- (UR1) -- cycle;
  \draw [fill opacity=0.7,fill=yellow] (UL1) -- (UL3) -- (UR3) -- (UR1) -- cycle;
  \draw (UR1) -- (UR2) -- (UR4) -- (UR3) -- cycle;
  \draw [fill opacity=0.7,fill=yellow] (UR1) -- (UR2) -- (UR4) -- (UR3) -- cycle;

  \coordinate (OL1) at (-1, 1,-1);
  \coordinate (OL2) at (-1, 1, 0);
  \coordinate (OL3) at (-1, 0,-1);
  \coordinate (OL4) at (-1, 0, 0);
  \coordinate (OR1) at ( 1, 1,-1);
  \coordinate (OR2) at ( 1, 1, 0);
  \coordinate (OR3) at ( 1, 0,-1);
  \coordinate (OR4) at ( 1, 0, 0);

  % green!80!blue
  \draw (OR1) -- (OR2) -- (OR4) -- (OR3) -- cycle;
  \draw [fill opacity=0.7,fill=orange] (OR1) -- (OR2) -- (OR4) -- (OR3) -- cycle;
  \draw (OL3) -- (OL4) -- (OR4) -- (OR3) -- cycle;
  \draw [fill opacity=0.7,fill=orange] (OL3) -- (OL4) -- (OR4) -- (OR3) -- cycle;
  \draw (OL1) -- (OL3) -- (OR3) -- (OR1) -- cycle;
  \draw [fill opacity=0.7,fill=orange] (OL1) -- (OL3) -- (OR3) -- (OR1) -- cycle;
  \draw (OL2) -- (OL4) -- (OR4) -- (OR2) -- cycle;
  \draw [fill opacity=0.7,fill=orange] (OL2) -- (OL4) -- (OR4) -- (OR2) -- cycle;
  \draw (OL1) -- (OL2) -- (OR2) -- (OR1) -- cycle;
  \draw [fill opacity=0.7,fill=orange] (OL1) -- (OL2) -- (OR2) -- (OR1) -- cycle;
  \draw (OL1) -- (OL2) -- (OL4) -- (OL3) -- cycle;
  \draw [fill opacity=0.7,fill=orange] (OL1) -- (OL2) -- (OL4) -- (OL3) -- cycle;

  \draw [fill] (0,0) circle [radius=0.1];

  % \draw [fill opacity=0.7,fill=orange!80!black] (A3) -- (A4) -- (B1) -- cycle;
  % \draw [fill opacity=0.7,fill=green!30!black] (A2) -- (A3) -- (C1) -- cycle;
  % \draw [fill opacity=0.7,fill=purple!70!black] (A3) -- (A4) -- (C1) -- cycle;
  \end{tikzpicture}%
  \quad\quad\quad&\quad\quad\quad
  \begin{tikzpicture}
  % [line join=bevel,z={( 3.85mm, -3.85mm)}]
  [line join=bevel,x={( 1.5cm, 0mm)},y={( 0mm, 1.5cm)},z={( 1.5*3.85mm, -1.5*3.85mm)}]

  \coordinate (OL1) at ( 0, 1,-1);
  \coordinate (OL2) at ( 0, 1, 0);
  \coordinate (OL3) at ( 0,-0.75,-1);
  \coordinate (OL4) at ( 0, 0, 0);
  \coordinate (OR1) at ( 1, 1,-1);
  \coordinate (OR2) at ( 1, 1, 0);
  \coordinate (OR3) at ( 1,-0.75,-1);
  \coordinate (OR4) at ( 1, 0, 0);

  % green!80!blue
  \draw (OR1) -- (OR2) -- (OR4) -- (OR3) -- cycle;
  \draw [fill opacity=0.7,fill=orange] (OR1) -- (OR2) -- (OR4) -- (OR3) -- cycle;
  \draw (OL3) -- (OL4) -- (OR4) -- (OR3) -- cycle;
  \draw [fill opacity=0.7,fill=orange] (OL3) -- (OL4) -- (OR4) -- (OR3) -- cycle;
  \draw (OL1) -- (OL3) -- (OR3) -- (OR1) -- cycle;
  \draw [fill opacity=0.7,fill=orange] (OL1) -- (OL3) -- (OR3) -- (OR1) -- cycle;
  \draw (OL2) -- (OL4) -- (OR4) -- (OR2) -- cycle;
  \draw [fill opacity=0.7,fill=orange] (OL2) -- (OL4) -- (OR4) -- (OR2) -- cycle;
  \draw (OL1) -- (OL2) -- (OR2) -- (OR1) -- cycle;
  \draw [fill opacity=0.7,fill=orange] (OL1) -- (OL2) -- (OR2) -- (OR1) -- cycle;
%   \draw (OL1) -- (OL2) -- (OL4) -- (OL3) -- cycle;
%   \draw [fill opacity=0.7,fill=orange] (OL1) -- (OL2) -- (OL4) -- (OL3) -- cycle;

  \coordinate (UL1) at ( 0,-0.75,-1);
  \coordinate (UL2) at ( 0,   -1,-1);
  \coordinate (UL3) at ( 1,-0.75,-1);
  \coordinate (UL4) at ( 1,   -1,-1);
  \coordinate (UR1) at ( 0, 0, 0);
  \coordinate (UR2) at ( 0,-1, 0);
  \coordinate (UR3) at ( 1, 0, 0);
  \coordinate (UR4) at ( 1,-1, 0);

  %orange!80!blue
  \draw (UL1) -- (UL2) -- (UL4) -- (UL3) -- cycle;
  \draw [fill opacity=0.7,fill=yellow] (UL1) -- (UL2) -- (UL4) -- (UL3) -- cycle;
  \draw (UL1) -- (UL3) -- (UR3) -- (UR1) -- cycle;
  \draw [fill opacity=0.7,fill=yellow] (UL1) -- (UL3) -- (UR3) -- (UR1) -- cycle;
  \draw (UL2) -- (UL4) -- (UR4) -- (UR2) -- cycle;
  \draw [fill opacity=0.7,fill=yellow] (UL2) -- (UL4) -- (UR4) -- (UR2) -- cycle;
  \draw (UL1) -- (UL2) -- (UR2) -- (UR1) -- cycle;
  \draw [fill opacity=0.7,fill=yellow] (UL1) -- (UL2) -- (UR2) -- (UR1) -- cycle;
  \draw (UL3) -- (UL4) -- (UR4) -- (UR3) -- cycle;
  \draw [fill opacity=0.7,fill=yellow] (UL3) -- (UL4) -- (UR4) -- (UR3) -- cycle;

%   \draw (UR1) -- (UR2) -- (UR4) -- (UR3) -- cycle;
%   \draw [fill opacity=0.7,fill=yellow] (UR1) -- (UR2) -- (UR4) -- (UR3) -- cycle;

  \coordinate (UL1) at ( 0, 0,0);
  \coordinate (UL2) at ( 0,-1,0);
  \coordinate (UL3) at ( 1, 0,0);
  \coordinate (UL4) at ( 1,-1,0);
  \coordinate (UR1) at ( 0,-0.75, 1);
  \coordinate (UR2) at ( 0,  -1, 1);
  \coordinate (UR3) at ( 1,-0.75, 1);
  \coordinate (UR4) at ( 1,-1, 1);

  %orange!80!blue
%   \draw (UL1) -- (UL2) -- (UL4) -- (UL3) -- cycle;
%   \draw [fill opacity=0.7,fill=yellow] (UL1) -- (UL2) -- (UL4) -- (UL3) -- cycle;
  \draw (UL2) -- (UL4) -- (UR4) -- (UR2) -- cycle;
  \draw [fill opacity=0.7,fill=yellow] (UL2) -- (UL4) -- (UR4) -- (UR2) -- cycle;
  \draw (UL1) -- (UL2) -- (UR2) -- (UR1) -- cycle;
  \draw [fill opacity=0.7,fill=yellow] (UL1) -- (UL2) -- (UR2) -- (UR1) -- cycle;
  \draw (UL3) -- (UL4) -- (UR4) -- (UR3) -- cycle;
  \draw [fill opacity=0.7,fill=yellow] (UL3) -- (UL4) -- (UR4) -- (UR3) -- cycle;
  \draw (UL1) -- (UL3) -- (UR3) -- (UR1) -- cycle;
  \draw [fill opacity=0.7,fill=yellow] (UL1) -- (UL3) -- (UR3) -- (UR1) -- cycle;
  \draw (UR1) -- (UR2) -- (UR4) -- (UR3) -- cycle;
  \draw [fill opacity=0.7,fill=yellow] (UR1) -- (UR2) -- (UR4) -- (UR3) -- cycle;

  \coordinate (OL1) at (-1, 1,-1);
  \coordinate (OL2) at (-1, 1, 0);
  \coordinate (OL3) at (-1, -0.75,-1);
  \coordinate (OL4) at (-1, 0, 0);
  \coordinate (OR1) at ( 0, 1,-1);
  \coordinate (OR2) at ( 0, 1, 0);
  \coordinate (OR3) at ( 0, -0.75,-1);
  \coordinate (OR4) at ( 0, 0, 0);

  % green!80!blue
%   \draw (OR1) -- (OR2) -- (OR4) -- (OR3) -- cycle;
%   \draw [fill opacity=0.7,fill=orange] (OR1) -- (OR2) -- (OR4) -- (OR3) -- cycle;
  \draw (OL1) -- (OL3) -- (OR3) -- (OR1) -- cycle;
  \draw [fill opacity=0.7,fill=orange] (OL1) -- (OL3) -- (OR3) -- (OR1) -- cycle;
  \draw (OL3) -- (OL4) -- (OR4) -- (OR3) -- cycle;
  \draw [fill opacity=0.7,fill=orange] (OL3) -- (OL4) -- (OR4) -- (OR3) -- cycle;
  \draw (OL2) -- (OL4) -- (OR4) -- (OR2) -- cycle;
  \draw [fill opacity=0.7,fill=orange] (OL2) -- (OL4) -- (OR4) -- (OR2) -- cycle;
  \draw (OL1) -- (OL2) -- (OR2) -- (OR1) -- cycle;
  \draw [fill opacity=0.7,fill=orange] (OL1) -- (OL2) -- (OR2) -- (OR1) -- cycle;
  \draw (OL1) -- (OL2) -- (OL4) -- (OL3) -- cycle;
  \draw [fill opacity=0.7,fill=orange] (OL1) -- (OL2) -- (OL4) -- (OL3) -- cycle;

  \draw [fill] (0,0) circle [radius=0.1];

  % \draw [fill opacity=0.7,fill=orange!80!black] (A3) -- (A4) -- (B1) -- cycle;
  % \draw [fill opacity=0.7,fill=green!30!black] (A2) -- (A3) -- (C1) -- cycle;
  % \draw [fill opacity=0.7,fill=purple!70!black] (A3) -- (A4) -- (C1) -- cycle;
  \end{tikzpicture}%
  \end{tabular} 
  
 \caption{
 Left: polyhedral domain in 3D
 that is not the graph of a Lipschitz function at the marked point.
 Right: bi-Lipschitz transformation of that domain into a strongly Lipschitz domain.
 }
 \label{fig:crossedbrickspicture}
\end{figure}
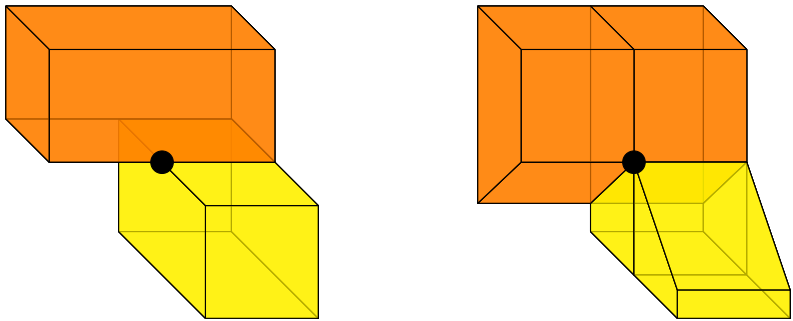

% \begin{example}
%  Every finitely-triangulated domain $\Omega$ in $\bbR^{3}$ is a weakly Lipschitz domain.
%  Let us outline the proof of that statement.
%  It is easy to obtain the relevant coordinate charts around $x \in \partial\Omega$
%  if $x$ is not a vertex of the triangulation. 
%  If instead $x$ is a vertex of the triangulation, 
%  then let $r > 0$ be so small that $B_r(x)$ intersects only tetrahedra adjacent to $x$. 
%  By  Theorem 7.8 of \cite{luukkainen1977elements}, there exists 
%  a lipeomorphism of $\partial B_r(x)$
%  which maps $\partial B_r(x) \cap \partial\Omega$ onto the intersection of $\partial B_r(x)$
%  and some plane through $x$. From this lipeomorphism on $\partial B_r(x)$,
%  we easily derive the desired coordinate chart over $B_r(x)$.
% \end{example}

The remainder of this section builds up a key notion of this article.
We show that weakly Lipschitz domains allow for a two-sided Lipschitz collar.
% This implies the existence of a tubular neighborhood of $\partial\Omega$
% in $\bbR^{n}$ which is parameterized transversally to the boundary in a Lipschitz manner.
This result will serve for the construction of a commuting extension operator
later in Section~\ref{sec:projection}.

\begin{theorem}
 \label{prop:closedtwosidedcollar}
 Let $\Omega \subseteq \bbR^{n}$ be a bounded weakly Lipschitz domain.
 Then there exists a LIP embedding 
 $\Psi : \partial\Omega \times [-1,1] \rightarrow \bbR^{n}$
 such that $\Psi(x,0) = x$ for $x \in \partial\Omega$, and
 \begin{align}
  \label{prop:closedtwosidedcollar:claimedinclusions}
  \Psi\left( \partial\Omega, [-1,0) \right) \subseteq \Omega,
  \quad
  \Psi\left( \partial\Omega, ( 0,1] \right) \subseteq \overline\Omega^{c}.
 \end{align}
\end{theorem}

\begin{proof}
 We first prove a one-sided version of the result.
 From definitions 
 we deduce that there exists a collection $\{V_i\}_{i \in \bbN}$
 of relatively open subsets of $\partial\Omega$
 that constitute a covering of $\partial\Omega$,
 and a collection $\{\psi_i\}_{i \in \bbN}$
 of LIP embeddings $\psi_i : V_i \times [0,1) \rightarrow \overline\Omega$
 such that for each $i\in \bbN$
 we have $\psi_i(x,0) = x$ for each $x \in \partial\Omega$.
 It follows that $\{ ( V_i, \psi_i ) \}_{i \in \bbN}$
 is a \emph{local LIP collar} in the sense of Definition 7.2 in \cite{luukkainen1977elements}.
 By Theorem 7.4 in \cite{luukkainen1977elements},
 and a successive reparametrizaton,
 there exists a LIP embedding 
 $\Psi^{-}(x,t) : \partial\Omega \times [0,1] \rightarrow \overline\Omega$
 such that $\Psi^{-}(x,0) = x$ for all $x \in \partial\Omega$.
 
%  Since $\Omega$ is a weakly Lipschitz domain,
%  we may pick for each $z \in \partial\Omega$ 
%  a closed neighborhood $U_z \subseteq \bbR^{n}$
%  and a bi-Lipschitz mapping $\phi_z : U_z \rightarrow [-1,1]^{n}$
%  which satisfy $\phi_z(z) = 0$ and \eqref{math:weaklylipschitzdomain}.
%  It follows that the closed sets $\{ \partial\Omega \cap U_z \st z \in \partial\Omega \}$
%  cover $\partial\Omega$.
%  We define mappings
%  \begin{align}
%   \Psi_z : \left( \partial\Omega \cap U_z \right) \times [-1,0] \rightarrow \overline\Omega,
%   \quad
%   (x,t) \mapsto \phi_z\left( \phi_z^{\inv}(x) - (0,\dots,0,t) \right)
%   .
%  \end{align}
%  It follows that $\{ \left( \partial\Omega \cap U_z, \Psi_z \right) \st z \in \partial\Omega \}$
%  is a \emph{local LIP collar} in the sense of Definition 7.2 in \cite{luukkainen1977elements}.
%  By Theorem 7.4 in \cite{luukkainen1977elements},
% %  and a successive reparametrizaton,
%  there exists a LIP embedding 
%  $\Psi^{-}(x,t) : \partial\Omega \times [0,1] \rightarrow \overline\Omega$
%  such that $\Psi^{-}(x,0) = x$ for all $x \in \partial\Omega$.
 
 Recall that also $\overline\Omega^{c}$ is a weakly Lipschitz domain. 
 By the same arguments, there exists a LIP embedding 
 $\Psi^{+}(x,t) : \partial\Omega \times [0,1] \rightarrow \Omega^{c}$
 such that $\Psi^{+}(x,0) = x$ for all $x \in \partial\Omega$.
 We combine these two LIP embeddings. 
 Let 
 \begin{align*}
  \Psi : \partial\Omega \times [-1,1] \rightarrow \bbR^{n},
  \quad
  (x,t) \mapsto
  \left\{ \begin{array}{rl}
  \Psi^{-}(x,-t) & \text{ if } t \in [-1,0), \quad x \in \partial\Omega,
  \\
  x & \text{ if } t = 0, \quad x \in \partial\Omega,
  \\
  \Psi^{+}(x, t) & \text{ if } t \in ( 0,1], \quad x \in \partial\Omega.
  \end{array} \right.
 \end{align*}
 Then $\Psi$ is well-defined, bijective, and
 \eqref{prop:closedtwosidedcollar:claimedinclusions} holds.
 Moreover, we have finite constants  
 \begin{gather*}
  C^{-} := \Lip( \Psi, \partial\Omega \times [-1,0] ),
  \quad 
  C^{+} := \Lip( \Psi, \partial\Omega \times [0,1] ),
  \\
  c^{-} := \Lip( \Psi^{\inv}, \Psi( \partial\Omega \times [-1,0] ) ),
  \quad 
  c^{+} := \Lip( \Psi^{\inv}, \Psi( \partial\Omega \times [0, 1] ) ).
 \end{gather*}
 It remains to show that $\Psi$ is a LIP embedding.
 Let $x_1, x_2 \in \partial\Omega$ and let $t_1, t_2 \in [-1,1]$.
 It suffices to show that 
 \begin{gather*}
  \| x_1 - x_2 \|
  +
  c \left| t_2 - t_1 \right|
  \leq 
  \left\| \Psi(x_1,t_1) - \Psi(x_2,t_2) \right\|
  \leq 
  \| x_1 - x_2 \|
  +
  C \left| t_2 - t_1 \right|
 \end{gather*}
 for $c = \max(c^{+},c^{-})^{\inv}$ and $C = \max( C^{+}, C^{-} )$.
 If $t_1$ and $t_2$ are both non-negative or both non-positive,
 then the both inequalities follow directly from the properties 
 of $\Psi^{+}$ or $\Psi^{-}$. Hence we consider 
 the case $t_1 < 0 < t_2$. 
 We first observe that 
 \begin{align*}
  \left\| \Psi(x_1,t_1) - \Psi(x_2,t_2) \right\|
  &\leq
  \left\| \Psi(x_1,t_1) - x_1 \right\|
  +
  \left\| x_1 - x_2 \right\|
  +
  \left\| x_2 - \Psi(x_2,t_2) \right\|
  \\&\leq
%   \left\| \Psi(x,s) - x \right\|
%   +
%   \left\| x - y \right\|
%   +
%   \left\| y - \Psi(y,t) \right\|
%   \\
%   &\quad\leq
  C^{-} |t_1| + C^{+} |t_2|
  +
  \| x_1 - x_2 \|
  \\&\leq
  \max(C^{+},C^{-}) \left| t_1 - t_2 \right|
  +
  \| x_1 - x_2 \|
  . 
 \end{align*}
 Furthermore, there exists $x \in \partial\Omega$ on the straight line segment 
 from $\Psi(x_1,t_1)$ to $\Psi(x_2,t_2)$. We then have 
 \begin{align*}
  \left\| \Psi(x_1,t_1) - \Psi(x_2,t_2) \right\|
  &=
  \left\| \Psi(x_1,t_1) - x \right\|
  +
  \left\| x - \Psi(x_2,t_2) \right\|
  \\&\geq 
  |t_1|
  +
  \left\| x_1 - x \right\|
  +
  \left\| x - x_2 \right\|
  +
  |t_2|
  \\&\geq 
  \max(c^{+},c^{-})^{\inv} |t_1 - t_2|
  +
  \left\| x_1 - x_2 \right\|
  . 
 \end{align*}
 This completes the proof. 
\end{proof}

\begin{remark}
 Our Theorem~\ref{prop:closedtwosidedcollar} realizes
 an idea from differential topology in a Lipschitz setting:
 if a surface is locally bi-collared,
 then it is also globally bi-collared.
 Such a result is well-known in the topological or smooth sense,
 but it seems to be only folklore in the Lipschitz sense. 
%  In the topological or smooth setting, that result is well-known,
%  but it seems to be only folklore in the Lipschitz setting.
 Notably, the result is mentioned in the unpublished preprint \cite{ghiloni2010complexity}.
 We have provided a proof for formal completeness.
% \end{remark}

% \begin{remark}
 For a strongly Lipschitz domain,
 it is well-known that a Lipschitz collar can be defined 
 using transversal vector fields near the boundary
 \cite{schoberl2008posteriori,StructPresDisc,gopalakrishnan2011partial}.
\end{remark}

% Let us assume that $\Omega$ is a bounded weakly Lipschitz domain.
% As a consequence of Theorem~\ref{prop:closedtwosidedcollar}
% and the compactness of $\partial\Omega$,
% there exist a compact neighborhood $\calC\Omega$ of $\partial\Omega$ in $\bbR^n$
% and a bi-Lipschitz mapping
% \begin{align*}
%  \Psi : \partial\Omega \times [-1,1] \rightarrow \calC\Omega
% \end{align*}
% such that $\Psi(x,0) = x$ for $x \in \partial\Omega$,
% and such that 
% \begin{gather*}
%  \Psi\left( \partial\Omega \times \{0\} \right) = \calC\Omega \cap \partial\Omega,
%  \\
%  \Psi\left( \partial\Omega \times [-1,0) \right) = \calC\Omega \cap \Omega,
%  \quad
%  \Psi\left( \partial\Omega \times (0,1] \right) = \calC\Omega \cap \overline\Omega^c.
%  % 
% \end{gather*}
% We call $\Psi$ the \emph{collar mapping}
% or just the \emph{Lipschitz collar}.
% We write
% \begin{align}
%  \label{math:definitionofcollarparts}
%  \calC^-\Omega := \calC\Omega \cap \Omega,
%  \quad
%  \calC^+\Omega := \calC\Omega \cap \overline\Omega^c,
%  \quad
%  \Omega^{e} = \overline\Omega \cup \calC^+\Omega 
% \end{align}
% for the \emph{interior collar part} $\calC^-\Omega$,
% the \emph{exterior collar part} $\calC^+\Omega$,
% and the \emph{extended domain} $\Omega^{e}$, respectively.
% Moreover, we have a well-defined bi-Lipschitz mapping
% \begin{align}
%  \label{math:definitioncoordinatereflection}
%  \calA : \calC^{+}\Omega \rightarrow \calC^{-}\Omega,
%  \quad
%  \Psi(x,t) \mapsto \Psi(x,-t)
% \end{align}
% from the outer collar part into the inner collar part,
% called \emph{collar reflection}.

% III. Differential forms 

\section{Differential forms}
\label{sec:differentialforms}

In this section we review the calculus of differential forms in a setting of low regularity.
Particular attention is given to differential forms with coefficients in $L^{p}$ spaces
and their transformation properties under bi-Lipschitz mappings. 
We adopt the notion of $L^{p,q}$ differential form of \cite{gol1982differential},
to which we also refer for further details on Lebesgue spaces of differential forms.
Further details can be found in \cite{gol1982differential},
from which the notion of $L^{p,q}$ differential form is adopted.
An elementary introduction to the calculus of differential forms is given in \cite{LeeSmooth}.
\\

Let $U \subseteq \bbR^n$ be an open set.
We let $M(U)$ denote the vector space of locally integrable functions over $U$. 
% up to equivalence almost everywhere.
For $k \in \bbZ$ we let $M\Lambda^{k}(U)$ be the vector space of locally integrable % measurable 
differential $k$-forms over $U$. % with respect to the Lebesgue measure.
% Note that $M(U) = M\Lambda^{0}(U)$.
% 
We denote by $\omega \wedge \eta \in M\Lambda^{k+l}(U)$ the exterior product of
$\omega \in M\Lambda^{k}(U)$ and $\eta \in M\Lambda^{l}(U)$,
and we recall that $\omega \wedge \eta = (-1)^{kl} \eta \wedge \omega$.
% Note that this product still has locally integrable coefficients,
% and that $\omega \wedge \eta = (-1)^{kl} \eta \wedge \omega$.

Let $e_1, \dots, e_n$ be the canonical orthonormal basis of $\bbR^{n}$.
% Let $e_1, \dots, e_n : U \rightarrow \bbR^{n}$ be the canonical constant coordinate vector fields.
The constant $1$-forms $\cartanx^1, \dots, \cartanx^n \in M\Lambda^{1}(U)$ 
are uniquely defined by $\cartanx^i(e_j) = \delta_{ij}$,
where $\delta_{ij} \in \{0,1\}$ denotes the Kronecker delta.
In the sequel, we let $\Sigma(k,n)$ denote the set of strictly ascending mappings
from $\{1,\dots,k\}$ to $\{1,\dots,n\}$.
Note that $\Sigma(0,n) = \{\emptyset\}$.
The \emph{basic $k$-alternators} are the exterior products    
\begin{align*}
 \cartanx^{\sigma} := \cartanx^{\sigma(1)} \wedge \dots \wedge \cartanx^{\sigma(k)} \in M\Lambda^{k}(U),
 \quad 
 \sigma \in \Sigma(k,n),
\end{align*}
and $\cartanx^{\emptyset} := 1$.
% We introduce a canonical representation of $k$-forms: 
Every $\omega \in M\Lambda^{k}(U)$ can be written uniquely as
\begin{align}
 \label{math:standardrepresentation_fields}
 \omega
 =
 \sum_{ \sigma \in \Sigma(k,n) }
 \omega_\sigma \cartanx^{\sigma}
%  \cartanx_{\sigma(1)} \wedge \ldots \wedge \cartanx_{\sigma(k)}
 ,
\end{align}
where $\omega_\sigma = \omega( e_{\sigma(1)}, \dots, e_{\sigma(k)} )$.
For every $n$-form $\omega \in M\Lambda^{n}(U)$ there exists a unique $\omega_n \in M(U)$
such that $\omega = \omega_n \vol^{n}_U$,
where $\vol^{n}_U := \cartanx^1 \wedge \ldots \wedge \cartanx^n$ is the canonical volume $n$-form of $\bbR^{n}$. 
We define the integral of $\omega \in M\Lambda^{n}(U)$ over $U$ as 
\begin{align}
 \label{math:integralofNform}
 \int_{U} \omega := \int_U \omega_n \; \dif x 
\end{align}
whenever $\omega_n \in M(U)$ is integrable.
If $\omega, \eta \in M\Lambda^{k}(U)$, 
then we define 
% $\langle \omega,\eta \rangle \in M(U)$ by 
\begin{align}
 \label{math:scalarproductofalternatingmultilinearforms}
 \langle \omega, \eta \rangle 
 :=
 \sum_{\sigma \in \Sigma(k,n)} \omega_{\sigma} \eta_{\sigma}
 \in 
 M(U)
 .
\end{align}
For $\omega \in M\Lambda^{k}(U)$ we let $|\omega| = \sqrt{\langle\omega,\omega\rangle}$. 
We let $L^{p}(U)$ denote the Lebesgue space with exponent $p \in [1,\infty]$,
and let $L^{p}\Lambda^{k}(U)$ denote the Banach space of differential $k$-forms
with coefficients in $L^{p}(U)$.
% Note that $L^{p}(U) = L^{p}\Lambda^{0}(U)$. 
The topology of $L^{p}\Lambda^{k}(U)$ is generated by the norm
\begin{align*}
 \left\| \omega \right\|_{L^{p}\Lambda^{k}(U)}
 :=
 \left\| \sqrt{ \langle \omega, \omega \rangle } \right\|_{L^{p}(U)}
 ,
 \quad 
 \omega \in L^{p}\Lambda^{k}(U)
 .
\end{align*}
We let $C\Lambda^{k}(\overline U)$ be the Banach space 
of continuous differential $k$-forms over $\overline U$, equipped with the maximum norm.
We let $C^{\infty}\Lambda^{k}(U)$ be the space of smooth differential $k$-forms over $U$,
we let $C^{\infty}\Lambda^{k}(\overline U)$ be the subspace of $C^{\infty}\Lambda^{k}(U)$
whose members can be extended smoothly onto $\bbR^{n}$,
and we let $C^{\infty}_{c}\Lambda^{k}(U)$ 
be the subspace of $C^{\infty}\Lambda^{k}(U)$
whose members have compact support in $U$. 
\\

The exterior derivative 
$\cartan : C^{\infty}\Lambda^{k}(\overline U) \rightarrow C^{\infty}\Lambda^{k+1}(\overline U)$
over smooth differential forms is defined by 
\begin{align}
 \label{math:exterior_derivative}
 \cartan \omega
 =
 \sum_{ \sigma \in \Sigma(k,n) }
 \sum_{ i=1 }^{n}
 ( \partial_i \omega_\sigma \cartanx^i ) \wedge  
 \cartanx^{\sigma},
 \quad
 \omega \in C^{\infty}\Lambda^{k}(\overline U),
\end{align}
where we use the representation \eqref{math:standardrepresentation_fields}. 
One can show that $\cartan$ is linear, that $\cartan\cartan = 0$, and that
\begin{align}
 \label{math:exteriorderivative_product}
 \cartan ( \omega \wedge \eta ) 
 =
 \cartan\omega \wedge \eta + (-1)^{k} \omega \wedge \cartan\eta
 , \quad 
 \omega \in C^{\infty}\Lambda^{k}(\overline U),
 \quad 
 \eta \in C^{\infty}\Lambda^{l}(\overline U).
\end{align}
We are are interested in defining the exterior derivative in a weak sense
over differential forms of low regularity.
If $\omega \in M\Lambda^{k}(U)$ and $\xi \in M\Lambda^{k+1}(U)$ such that
\begin{align}
 \label{math:distributionalexteriorderivative}
 \int_U \xi \wedge \eta
 =
 (-1)^{k+1} \int_U \omega \wedge \cartan\eta,
 \quad 
 \eta \in C_c^{\infty}\Lambda^{n-k-1}(U),
\end{align}
then $\xi$ is the only member of $M\Lambda^{k+1}(U)$ with this property,
up to equivalence almost everywhere, and we call $\cartan \omega := \xi$
the \emph{weak exterior derivative} of $\omega$.
Note that $\cartan\omega$ has vanishing weak exterior derivative, since 
\begin{align}
 \int_U \cartan\omega \wedge \cartan\eta
 =
 (-1)^{k} \int_U \omega \wedge \cartan\cartan\eta
 = 0,
 \quad 
 \eta \in C_c^{\infty}\Lambda^{n-k-1}(U).
\end{align}
Moreover, \eqref{math:exteriorderivative_product} generalizes in the obvious manner 
to the weak exterior derivative, provided all expressions are well-defined.
\\

Next we introduce a notion of Sobolev differential forms.
For $p, q \in [1,\infty]$, 
we let $L^{p,q}\Lambda^{k}(U)$ be the space of differential $k$-forms in $L^{p}\Lambda^{k}(U)$
whose members have a weak exterior derivative in $L^{q}\Lambda^{k+1}(U)$.
We equip $L^{p,q}\Lambda^{k}(U)$ with the norm
\begin{align}
 \| \omega \|_{L^{p,q}\Lambda^{k}(U)}
 =
 \| \omega \|_{L^{p}\Lambda^{k}(U)}
 +
 \| \cartan \omega \|_{L^{q}\Lambda^{k+1}(U)}.
\end{align}
% It is obvious that ${L^{p,q}\Lambda^{k}(U)}$ is a Banach space.
% The following two approximation results,
% which are stated in \cite[Lemma 1.3]{gol1982differential}, 
% are of interest to us in the sequel:
Moreover, we have the following density result \cite[Lemma 1.3]{gol1982differential}.

\begin{lemma}
 \label{prop:approximation:Lpq}
 If $p,q \in [1,\infty)$,
 then $C^{\infty}\Lambda^{k}(\overline U)$ is dense in $L^{p,q}\Lambda^{k}(U)$.
\end{lemma}

% % % % \begin{lemma}
% % % %  \label{prop:approximation:flat}
% % % %  For $\omega \in L^{\infty,\infty}\Lambda^{k}(U)$
% % % %  there exists a sequence $\omega_i \in C^{\infty}\Lambda^{k}(\overline U)$
% % % %  such that for compact subsets $K \subset U$ we have
% % % %  $\omega_i \rightarrow \omega$ and $\cartan\omega_i \rightarrow \cartan\omega$
% % % %  almost everywhere in $K$.
% % % % %  $\omega_i \rightarrow \omega$ in $L^{\infty}\Lambda^{k}(K)$
% % % % %  and $\cartan\omega_i \rightarrow \cartan\omega$ in $L^{\infty}\Lambda^{k+1}(K)$.
% % % % \end{lemma}

Note that, by definition, $\cartan L^{p,q}\Lambda^{k}(U) \subset L^{q,r}\Lambda^{k+1}(U)$
for $p,q,r \in [1,\infty]$. 
Hence one may study de~Rham complexes of the form
\begin{align*}
 \begin{CD}
  \cdots 
  @>\cartan>>
  L^{p,q}\Lambda^{k  }(U)
  @>\cartan>>
  L^{q,r}\Lambda^{k+1}(U)
  @>\cartan>>
  \cdots 
 \end{CD}
\end{align*}
The choice of the Lebesgue exponents determines
analytical and algebraic properties of these de~Rham complexes.
This is not a subject of the present article,
but we refer to \cite{troyanov2006sobolev} for corresponding results over smooth manifolds.
De~Rham complexes of the above form with a Lebesgue exponent $p$ fixed 
are known as $L^{p}$ de~Rham complexes (e.g.\ \cite{mitrea2002traces}).
% where the integrability coefficients may chosen arbitrarily.
Two examples of such de~Rham complexes
% , which have received much attention in literature,
are of specific relevance to us.

\begin{example}
 The space $H\Lambda^{k}(U) = L^{2,2}\Lambda^{k}(U)$
 is a Hilbert space, consisting of those $L^{2}$ differential $k$-forms
 whose exterior derivative has $L^{2}$ integrable coefficients.
%  The scalar product
%  on $H\Lambda^{k}(U)$ is given by  
 \begin{align*}
  \langle \omega, \eta \rangle_{H\Lambda^{k}(U)}
  =
  \langle \omega, \eta \rangle_{L^{2}\Lambda^{k}(U)}
  +
  \langle \cartan \omega, \cartan \eta \rangle_{L^{2}\Lambda^{k+1}(U)},
  \quad
  \omega, \eta \in H\Lambda^{k}(U).
 \end{align*}
 induces the topology of $H\Lambda^{k}(U)$. 
 In particular, the norms $\| \cdot \|_{L^{2,2}\Lambda^{k}(U)}$
 and $\| \cdot \|_{H\Lambda^{k}(U)}$ are equivalent.
 These spaces constitute the $L^{2}$ de~Rham complex 
 \begin{align*}
  \begin{CD}
   \cdots 
   @>\cartan>>
   H\Lambda^{k  }(U)
   @>\cartan>>
   H\Lambda^{k+1}(U)
   @>\cartan>>
   \cdots 
  \end{CD}
 \end{align*}
 which has received considerable attention in global and numerical analysis.
\end{example}

\begin{example}
 The space $L^{\infty,\infty}\Lambda^{k}(U)$ of \emph{flat differential forms} 
 is spanned by those differential forms with essentially bounded coefficients whose exterior derivative 
 has essentially bounded coefficients.
 This coincides with the notion of flat differential form in geometric integration theory \cite{whitney2012geometric}; 
 see also Theorem 1.5 of \cite{gol1982differential}.
 These spaces constitute the flat de~Rham complex 
 \begin{align*}
  \begin{CD}
   \cdots 
   @>\cartan>>
   L^{\infty,\infty}\Lambda^{k  }(U)
   @>\cartan>>
   L^{\infty,\infty}\Lambda^{k+1}(U)
   @>\cartan>>
   \cdots 
  \end{CD}
 \end{align*}
 which has been studied extensively in geometric integration theory.
\end{example}

We conclude this section with some basic results on the behavior of differential forms
and their integrals under pullback by bi-Lipschitz mappings.
For the remainder of this section, we let $U, V \subseteq \bbR^{n}$ be open sets, 
and let $\Phi  : U \rightarrow V$ be a bi-Lipschitz mapping.
% So $\Phi ^{\inv} : V \rightarrow U$ is Lipschitz.

% Let $\Phi  : U \rightarrow V$ be a bi-Lipschitz mapping between open subsets $U, V \subseteq \bbR^{n}$.
We first gather some facts on the Jacobians of bi-Lipschitz mappings. 
It follows from Rademacher's theorem \cite[Theorem 3.1.6]{federer2014geometric} that the Jacobians
\begin{align*}
 \Jacobian \Phi  : U \rightarrow \bbR^{n\times n},
 \quad
 \Jacobian \Phi ^{\inv} : V \rightarrow \bbR^{n\times n}
\end{align*}
exist almost everywhere. One can show that 
\begin{align}
 \label{math:lipschitzdifferential:bounds}
 \| \Jacobian\Phi  \|_{L^{\infty}(U)} \leq \Lip(\Phi ,U),
 \quad
 \| \Jacobian\Phi ^{\inv} \|_{L^{\infty}(V)} \leq \Lip(\Phi ^{\inv},V).
\end{align}
According to \cite[Lemma 3.2.8]{federer2014geometric}, the identities
\begin{align}
 \label{math:lipschitzdifferential:inversejacobians}
 \Jacobian \Phi ^{\inv}_{\Phi (x)} \cdot \Jacobian \Phi _{x}
 =
 \Id_{U},
 \quad
 \Jacobian \Phi _{\Phi ^{\inv}(y)} \cdot \Jacobian \Phi ^{\inv}_{y}
 =
 \Id_{V}
\end{align}
hold true almost everywhere over $U$ and $V$, respectively.
In particular, the Jacobians have full rank almost everywhere. 
Moreover, by \cite[Corollary 4.1.26]{federer2014geometric} 
the signs of the Jacobians are essentially locally constant:
under the condition that $U$ and $V$ are simply connected, 
there exists $o(\Phi ) \in \{-1,1\}$ such that
\begin{align}
 \label{math:lipschitzdifferential:signum}
 o(\Phi ) = \sgn \det \Jacobian\Phi ,
 \quad
 o(\Phi )
 = o(\Phi ^{\inv})
\end{align}
almost everywhere over $U$, respectively. 
It follows from \cite[Theorem 3.2.3]{federer2014geometric} that
\begin{align}
 \label{math:transformationssatz}
 \int_{U} \omega \left( \Phi (x) \right) | \det\Jacobian\Phi_{x}  | \;\dif x
 =
 \int_{V} \omega (y) \;\dif y 
\end{align}
for $\omega  \in M(V)$ if at least one of the integrals exists.

The pull-back $\Phi ^{\ast} \omega \in M\Lambda^{k}(U)$ of $\omega \in M\Lambda^{k}(V)$ under $\Phi$ is defined as  
\begin{align*}
 \Phi ^{\ast} \omega_x( v_1, \dots, v_k )
 :=
 \omega_{\Phi (x)} ( \Jacobian\Phi _x \cdot v_1, \dots, \Jacobian\Phi _x \cdot v_k ).
\end{align*}
By the discussion at the beginning of Section 2 of \cite{gol1982differential},
the algebraic identity
\begin{align*}
 \Phi ^{\ast} ( \omega \wedge \eta ) = \Phi ^{\ast} \omega \wedge \eta + (-1)^{k} \omega \wedge \Phi ^{\ast} \eta 
\end{align*}
holds for $\omega \in M\Lambda^{k}(V)$ and $\eta \in M\Lambda^{l}(V)$. 
Next we show how the integral of $n$-forms transforms under pullback by bi-Lipschitz mappings:

\begin{lemma}
 \label{prop:transformationssatz}
 If $\Phi  : U \rightarrow V$ is a bi-Lipschitz mapping between simply connected open subsets of $\bbR^{n}$,
 then
 \begin{align}
  \label{math:volumeform:transformationssatz}
  \int_{U} \Phi ^{\ast} \left( \omega \vol^{n}_V \right)  
  =
  o( \Phi ) \int_{V} \omega \vol^{n}_V,
  \quad
  \omega \in M(V).
 \end{align}
\end{lemma}
\begin{proof}
 Using \eqref{math:lipschitzdifferential:signum} and \eqref{math:transformationssatz},
 we find 
 \begin{align*}
  \int_{U} \Phi ^{\ast} \left( \omega \vol^{n}_V \right)  
  &=
  \int_{U} \omega \circ \Phi (x) \det\left( \Jacobian\Phi _x \right) \dif x
%   \\&=
%   \int_{U}
%   \omega \circ \Phi (x) \cdot
%   \det\Jacobian\Phi _{\Phi ^{\inv}\Phi (x)} \cdot
%   |\det\Jacobian\Phi _{\Phi ^{\inv}\Phi (x)}|^{\inv}
%   |\det\Jacobian\Phi _{x}|
%   \dif x
  \\&=
  \int_{U}
  \omega \circ \Phi (x) \cdot
  \signum\det\Jacobian\Phi _{x} \cdot
  |\det\Jacobian\Phi _{x}|
  \; \dif x
% %  \end{align*}
% %  Applying \eqref{math:lipschitzdifferential:signum} and \eqref{math:transformationssatz},
% %  we observe 
% %  \begin{align*}
%   &\int_{U}
%   \omega \circ \Phi (x) \cdot
%   \signum\det\Jacobian\Phi _{\Phi ^{\inv}\Phi (x)} \cdot
%   |\det\Jacobian\Phi _{x}|
%   \dif x
  =
%   \int_{V}
%   \omega \cdot
%   \signum\det\Jacobian\Phi _{\Phi ^{\inv}(x)}
%   \dif x
%   =
  o( \Phi ) \int_{V} \omega \vol^{n}_V
  .
 \end{align*}
 This shows the desired identity.
\end{proof}

It can be shown that the pullback under bi-Lipschitz mappings
commutes with the exterior derivative 
and preserves the $L^{p}$ and $L^{p,q}$ classes of differential forms.
% 
% Besides the integration of $n$-forms, we are also interested in integral estimates that pertain to the $L^{p}$ norm
% of differential forms. Here and in the sequel, we set $n/\infty := 0$ for $n \in \bbN$.
% The pullback by bi-Lipschitz mappings preserves $L^{p}$ regularity, and commutes with the exterior derivative.

\begin{lemma}[Theorem 2.2 of \cite{gol1982differential}]
 \label{prop:pullback:Lpq}
 Let $\Phi  : U \rightarrow V$ be a bi-Lipschitz mapping 
 between simply-connected open subsets of $\bbR^{n}$,
 and let $p,q \in [1,\infty]$.
 If $\omega \in L^{p}\Lambda^{k}(V)$, then $\Phi ^{\ast} \omega \in L^{p}\Lambda^{k}(U)$,
 and if moreover $\omega \in L^{p,q}\Lambda^{k}(V)$, then $\Phi ^{\ast} \omega \in L^{p,q}\Lambda^{k}(U)$
 and $\Phi ^{\ast} \cartan \omega = \cartan \Phi ^{\ast} \omega$.
\end{lemma}

We refine the preceding statement and give an explicit estimate for the norm of the pullback operation. 
Here and in the sequel, $n/\infty = 0$ for $n \in \bbN$.

\begin{theorem}
 \label{prop:pullbackestimate}
 Let $\Phi : U \rightarrow V$ be a bi-Lipschitz mapping
 between open sets $U, V \subseteq \bbR^{n}$,
 and let $p \in [1,\infty]$. 
 Then 
 \begin{align}
  \label{math:pullbackestimate}
  \begin{split}
  \| \Phi^{\ast} \omega \|_{L^{p}\Lambda^{k}(U)}
  &\leq 
  \| \Jacobian\Phi \|^{k}_{L^{\infty}(U)}
  \| \det\Jacobian\Phi^{\inv} \|^{\frac{1}{p}}_{L^{\infty}(V)}
  \| \omega \|_{L^{p}\Lambda^{k}(V)}
  \\&\leq 
  \| \Jacobian\Phi \|^{k}_{L^{\infty}(U)}
  \| \Jacobian\Phi^{\inv} \|^{\frac{n}{p}}_{L^{\infty}(V)}
  \| \omega \|_{L^{p}\Lambda^{k}(V)}
  \end{split}
 \end{align}
 for $\omega \in L^{p}\Lambda^{k}(U)$.
\end{theorem}

\begin{proof}
 Let $\Phi : U \rightarrow V$ and $p \in [1,\infty]$
 be as in the statement of the theorem,
 and let $\omega \in L^{p}\Lambda^{k}(U)$.
 If $x \in U$ such that $\Jacobian\Phi_{|x}$ exists, then we observe
 \begin{align*}
%   | \Phi^{\ast} \omega |^{2}_{|x}
%   &=
  \left\| \Phi^{\ast} \omega \right\|^{2}_{\ell^{2}|x}
  &=
  \sum_{ \rho \in \Sigma(k,n) }
  \left\langle \Phi^{\ast} \omega_{|x}, \cartanx^{\rho} \right\rangle^{2} 
%   \\&=
%   \sum_{ \sigma, \rho \in \Sigma(k,n) } 
%   \left( \Phi^{\ast} \omega_{\sigma|x} \right)^{2} 
%   \cdot \left\langle \Phi^{\ast} \cartanx^{\sigma}_{|x}, \cartanx^{\rho} \right\rangle^{2} 
  \\&=
  \sum_{ \sigma \in \Sigma(k,n) } 
  \left( \omega_{\sigma|\Phi(x)} \right)^{2} 
  \sum_{ \rho \in \Sigma(k,n) } 
  \left\langle \Phi^{\ast} \cartanx^{\sigma}_{|x}, \cartanx^{\rho} \right\rangle^{2} 
  \\&=
  \sum_{ \sigma \in \Sigma(k,n) } 
  \left( \omega_{\sigma|\Phi(x)} \right)^{2} 
  \left\| \Phi^{\ast} \cartanx^{\sigma}_{|x} \right\|^{2} 
  \\&\leq
  \left\| \Jacobian\Phi \right\|^{2k} 
  \sum_{ \sigma \in \Sigma(k,n) } 
  \left( \omega_{\sigma|\Phi(x)} \right)^{2} 
  =
  \left\| \Jacobian\Phi_{|x} \right\|^{2k} 
  \left\| \omega_{|\Phi(x)} \right\|^{2} 
  . 
 \end{align*}
 From this we easily infer that 
 \begin{align*}
  \| \Phi^{\ast} \omega \|_{L^{p}\Lambda^{k}(U)}
  \leq
  \| \Jacobian\Phi \|^{k}_{L^{\infty}(U)}
  \big\| \|\omega\|_{\ell^{2}} \circ \Phi \big\|_{L^{p}\Lambda^{k}(U)}
  . 
 \end{align*}
 If $p = \infty$, then the desired statement follows trivially,
 and if $p \in [1,\infty)$, then Lemma~\ref{prop:transformationssatz} provides 
 \begin{align*}
  \int_{U} \left\| \omega \right\|^{p}_{|\Phi(x)} \dif x
  &\leq 
  \| \det \Jacobian\Phi_{|x}^{\inv} \|_{L^{\infty}(V)}
  \int_{U}
   \left\| \omega \right\|^{p}_{|\Phi(x)} 
   |\det \Jacobian\Phi_{|x}|
  \dif x
  \\&\leq 
  \| \det \Jacobian\Phi_{|x}^{\inv} \|_{L^{\infty}(V)}
  \int_{\Phi(U)}
   \left\| \omega \right\|^{p}_{x} 
  \dif x
 \end{align*}
 We note that $\Phi(U) = V$. This shows the first inequality,
 and the second inequality follows by Hadamard's inequality.
\end{proof}

% IV. Triangulations

\section{Triangulations}
\label{sec:triangulations}

In this section we review simplicial triangulations of domains and related notions,
of which most is standard in literature. 
% We consider only simplicial discretizations of domains in this article. 

%%%%%
%%%%% Introduce triangulation
%%%%%

Let us assume that $\Omega \subseteq \bbR^{n}$ is a bounded open set
such that $\overline\Omega$ is a topological manifold with boundary
and $\Omega$ is its interior.
A \textit{finite triangulation} of $\overline\Omega$ is a finite set $\calT$ of closed simplices
such that the union of the elements of $\calT$ equals $\overline \Omega$,
such that for any $T \in \calT$ and any subsimplex $S \subseteq T$ we have $S \in \calT$,
and such that for all $T, T' \in \calT$ the set $T\cap T'$ is either empty
or a common subsimplex of both $T$ and $T'$.
We write
\begin{align*}
 \Delta(T) := \left\{ S \in \calT \st S \subseteq T \right\}
 ,
 \quad
 \calT(T) := \left\{ S \in \calT \st S \cap T \neq \emptyset \right\}
 .
\end{align*}
With some abuse of notation, we let $\calT(T)$ also denote the closed set 
that is the union of the simplices of $\calT$ adjacent to $T$.
We write $\calT^m$ for the set of $m$-dimensional simplices in $\calT$.

% Having formally introduced triangulations, 
% we make precise and prove that polyhedral domains in $\bbR^{3}$,
% in the following sense, are weakly Lipschitz domains.
Having formally introduced triangulations,
we make precise and prove the introduction's claim 
that all polyhedral domains in $\bbR^{3}$
are weakly Lipschitz domains. 

\begin{theorem}
 Let $\Omega \subset \bbR^{3}$ be an open set
 such that $\overline\Omega$ is a topological manifold with interior $\Omega$.
 If there exists a finite triangulation $\calT$ of $\Omega$,
 then $\Omega$ is a weakly Lipschitz domain. 
\end{theorem}
% FIXME: Decompress this proof.
\begin{proof}
 Let $x \in \partial\Omega$. 
 We need to find a compact neighborhood $U_x \subseteq \bbR^{3}$ of $x$
 and a bi-Lipschitz mapping $\phi_x : U_x \rightarrow [-1,1]^3$
 such that $\phi_x(x) = 0$ and \eqref{math:weaklylipschitzdomain}.
 It is easy to find such $U_x$ and $\phi_x$ 
 if $x \notin \calT^{0}$.
 
 It remains to consider the case $x \in \calT^{0}$.
 Let $r > 0$ be so small that $B_r(x)$
 intersects $T \in \calT^{3}$ if and only if $x \in T$.
 We observe that $\partial B_r(x) \cap \partial\Omega$
 is a simple closed curve in $\partial B_r(x)$
 composed of finitely many spherical arcs.
 By Theorem 7.8 of \cite{luukkainen1977elements},
 there exists a bi-Lipschitz mapping
 \begin{align*}
  \phi^{0}_x : \partial B_r(x) \rightarrow \partial B_1(0) \subset \bbR^{3}
 \end{align*}
%  $\phi^{0}_x : \partial B_r(x) \rightarrow \partial B_1(0) \subset \bbR^{3}$
 which maps $\partial B_r(x) \cap \partial\Omega$ onto $\partial B_1(0) \cap \{x \in \bbR^{3} \st x_3 = 0\}$.
 By radial continuation, we obtain a bi-Lipschitz mapping 
 \begin{align*}
  \phi^{1}_x : B_r(x) \rightarrow B_1(0) \subset \bbR^{3}
 \end{align*}
% $\phi^{1}_x : B_r(x) \rightarrow B_1(0) \subset \bbR^{3}$
 which maps $B_r(x) \cap \partial\Omega$ onto $B_1(0) \cap \{x \in \bbR^{3} \st x_3 = 0\}$
 and with $\phi^{1}_x(x) = 0$.
 Moreover, there exists a bi-Lipschitz mapping 
 \begin{align*}
  \phi^{2}_x : B_1(0) \rightarrow [-1,1]^{3}
 \end{align*}
% $\phi^{2}_x : B_1(0) \rightarrow [-1,1]^{3}$ be a bi-Lipschitz mapping 
 which maps $B_1(0) \cap \{x \in \bbR^{3} \st x_3 = 0\}$ onto $[-1,1]^{2} \times \{0\}$
 with $\phi^{2}_x(0) = 0$.
 The theorem follows with $U_x := B_r(x)$ and $\phi_x := \phi_x^{2} \phi_x^{1}$.
\end{proof}

%%%%%
%%%%% shape stuff 
%%%%%

The remainder of this section is devoted to notions of regularity of triangulations.
Let us fix a finite triangulation $\calT$ of $\overline\Omega$.
When $T \in \calT^{m}$ is any simplex of the triangulation, 
then we write $h_T = \diam(T)$ for the diameter of $T$,
and $|T| = \vol^{m}(T)$ for the $m$-dimensional volume of $T$.
If $V \in \calT^{0}$, then $|V| = 1$,
and $h_V$ is defined, by convention, as the average length of all $n$-simplices of $\calT$
that are adjacent to $V$.
% Moreover, we write $h_{\Omega} = \diam(\Omega)$.

% Let us revisit some notions of regularity of triangulations 
% without going too much into theoretical foundations.
We define the \emph{shape constant} of $\calT$ as the minimal $C_{mesh} > 0$ 
that satisfies 
% \footnote{Decide whether this should restrict to $n$-simplices.}
% be the \emph{shape constant} of $\calT$, 
% which is the minimal number that satisfies 
\begin{gather}
 \label{math:shapeconditions:flatness}
 \forall T \in \calT^{n}
 : h_T^n \leq C_{mesh} |T|,
 \\
 \label{math:shapeconditions:comparison}
 \forall T \in \calT, S \in \calT(T)
 : 
 h_T \leq C_{mesh} h_S.
\end{gather}
Intuitively, \eqref{math:shapeconditions:flatness} describes a bound on the flatness of the simplices,
while \eqref{math:shapeconditions:comparison} describes that the diameter of adjacent simplices are comparable.
In applications, we consider families of triangulations,
such as generated by successive uniform refinement \cite{Bey}
or newest vertex bisection \cite{maubach1995local},
whose shape constants are uniformly bounded.
% satisfy \eqref{math:shapeconditions:flatness} and \eqref{math:shapeconditions:comparison}
% with uniformly bounded shape constants. 
% 
% with uniformly bounded shape constants.
% For example, sequences of triangulations generated by successive uniform refinement \cite{Bey}
% or newest vertex bisection \cite{maubach1995local}
% % have uniformly bounded shape constants.
% satisfy \eqref{math:shapeconditions:flatness} and \eqref{math:shapeconditions:comparison}
% with uniformly bounded constants. 

% We henceforth call a real number \emph{shape-uniform} or \emph{shape-uniformly bounded}
% if it can be bounded in terms of $C_{mesh}$ and the dimension $n$.
We can bound some important quantities in terms of $C_{mesh}$ and the geometric ambient. 
There exists a constant $C_N > 0$, depending only on $C_{mesh}$ and the ambient dimension $n$, such that 
\begin{align}
 \forall T \in \calT 
 :
 |\calT(T)| \leq C_{N}.
\end{align}
This bounds the numbers of neighbors of any simplex. 
There exists a constant $\eps_{h} > 0$, depending only on $C_{mesh}$ and $\Omega$, such that
\begin{subequations}
\label{math:neighborcontainment}
\begin{gather}
 \label{math:neighborcontainment:uno}
 \forall T \in \calT 
 :
 B_{\eps_h h_T}(T) \subseteq \Omega^{e},
 \\
 \label{math:neighborcontainment:duo}
 \forall T \in \calT 
 :
 B_{\eps_h h_T}(T) \cap \overline\Omega \subseteq \calT(T).
\end{gather}
\end{subequations}
In the sequel, we use affine transformations to a reference simplex. Let 
\begin{align*}
 \Delta^{n} = \convex\{0,e_1,\dots,e_n\} \subseteq \bbR^{n}
\end{align*}
be the $n$-dimensional reference simplex.
For each $n$-simplex $T \in \calT^{n}$ of the triangulation,
we fix an affine transformation
$A_T(x) = M_T x + b_T$
where $b_T \in \bbR^{n}$ and $M_T \in \bbR^{n \times n}$ 
are such that $A_T( \Delta^{n} ) = T$.
Each matrix $M_T$ is invertible, and 
% we have 
% for each $T \in \calT^{n}$,
\begin{align}
 \label{math:referencetransformation}
 \| M_T \|
 \leq c_{M} h_T, 
 \quad 
 \| M_T^{\inv} \|
 \leq C_{M} h_T^{\inv}
\end{align}
for constants $c_{M}, C_{M} > 0$ that depend only on $C_{mesh}$ and $n$.

% V. Chains 

\section{Elements of Geometric Measure Theory}
\label{sec:gmt}

This section gives an outline of relevant ideas from geometric measure theory,
for which we use Whitney's monograph \cite{whitney2012geometric} as the main reference. 
Our motivation to consider geometric measure theory
lies in proving Theorem~\ref{prop:interpolation_error} later in this paper.
A key observation for this purpose is that the degrees of freedom 
of finite element exterior calculus are \emph{flat chains} (Lemma~\ref{prop:dof_are_chains}).
Analogously, we identify finite element differential forms as \emph{flat forms}.
This allows us to estimate Lipschitz deformations of flat chains (Lemma~\ref{prop:deformationestimate})
in a finite element setting. 
\\

We begin with basic notions of chains and cochains,
The space $P^{k}(\bbR^{n})$ of \emph{polyhedral $k$-chains}
is the vector space of functions 
which can be written as finite sums $\sum_{i=1}^{l} a_i \chi_{S_i}$,
where $a_i \in \bbR$ and $\chi_{S_i}$ is the indicator function of a closed simplex $S_i$.
We consider two such functions as equivalent if they are identical almost everywhere
with respect to the $k$-dimensional Hausdorff measure. 

We fix an arbitrary orientation for each $k$-simplex $S \subseteq \bbR^{n}$,
henceforth called positive orientation. We identify $S$ in positive orientation with $\chi_S$,
and $S$ in negative orientation with $-\chi_S$.
The \emph{boundary operator} $\partial : P^{k}(\bbR^{n}) \rightarrow P^{k-1}(\bbR^{n})$
is now defined as follows. If $S$ is an oriented $k$-simplex,
then $\partial S$ is the formal sum $F_0 + \dots + F_{k}$
of its faces equipped with the outward orientation induced by $S$. 
This defines $\partial$ over $P^{k}(\bbR^{n})$ by linear extension.

The \emph{mass} $|S|_{k}$ of a polyhedral $k$-chain is its $L^{1}$-norm 
with respect to the $k$-dimensional Hausdorff measure. 
We write $\overline{P^{k}({\bbR^n})}$ for the completion of $P^{k}({\bbR^n})$ by the mass norm.
We define the \emph{flat norm} of $S \in P^{k}({\bbR^n})$ as
\begin{align}
 \label{math:definition_flat_norm}
 \| S \|_{k,\flat} 
 := \inf_{ Q \in P^{k+1}({\bbR^n}) }
%  \left\{ 
  |S - \partial Q|_{k} + |Q|_{k+1} 
%   )
%  \right\}.
 .
\end{align}
One can show that $\|\cdot\|_{k,\flat}$ is a norm on $P^{k}({\bbR^n})$.
We define the Banach space $\calC_{k}^{\flat}({\bbR^n})$,
the space of \emph{flat $k$-chains in ${\bbR^n}$}
as the completion of $P^{k}({\bbR^n})$ with respect to the flat norm.
We have $\| S \|_{k,\flat} \leq | S |_{k}$ for $S \in P^{k}({\bbR^n})$,
so $\overline{P^{k}(\bbR^{n})}$ embeds densely into $\calC_{k}^{\flat}({\bbR^n})$.
We moreover have $\| \partial S \|_{k-1,\flat} \leq \| S \|_{k,\flat}$ for $S \in P^{k}({\bbR^n})$,
so the boundary operator extends to a bounded linear operator
$\partial : \calC_{k}^{\flat}({\bbR^n}) \rightarrow \calC_{k-1}^{\flat}({\bbR^n})$.
Note, however, that $\partial$ is densely-defined but unbounded over $\overline{P^{k}({\bbR^n})}$.
\\

We now study the duality of flat chains and flat differential forms. 
Flat forms were studied in \cite{whitney2012geometric},
there mainly as representations of flat cochains,
and in \cite{gol1982differential}.
For the following facts, we refer to 
Section 2 of \cite{gol1982differential} and
Chapters IX and X of \cite{whitney2012geometric}.

Flat differential forms have well-defined traces on simplices.
More precisely, for each $m$-simplex $S \subset \bbR^{n}$
there exists a bounded linear mapping 
$\trace : L^{\infty,\infty}\Lambda^{k}(\bbR^{n}) \rightarrow L^{\infty,\infty}\Lambda^{k}(S)$,
which extends the trace of smooth forms.
In particular, $\trace_{S} \omega$ does only depend on the values of $\omega$ near $S$.
We write $\int_S \omega$ for the integral of $\omega \in L^{\infty,\infty}\Lambda^{k}(\bbR^{n})$
over a $k$-simplex $S$. 
This pairing extends by linearity to $S \in \overline{P^{k}(\bbR^{n})}$. 
We have 
\begin{align}
 \label{math:dualityestimate:polyhedral}
 \left| \int_S \omega \right|
 \leq 
 | S |_{k} \| \omega \|_{L^{\infty,\infty}\Lambda^{k}({\bbR^n})},
 \quad 
 \omega \in L^{\infty,\infty}\Lambda^{k}({\bbR^n}),
 \quad 
 S \in \overline{P^{k}(\bbR^{n})}.
\end{align}
This pairing furthermore extends to flat chains. We have 
\begin{align}
 \label{math:dualityestimate:flat}
 \left| \int_\alpha \omega \right|
 \leq 
 \| \alpha \|_{k,\flat} \| \omega \|_{L^{\infty,\infty}\Lambda^{k}({\bbR^n})},
 \quad 
 \alpha \in \calC^{\flat}_{k}(\bbR^{n}),
 \quad 
 \omega \in L^{\infty,\infty}\Lambda^{k}({\bbR^n}).
\end{align}
The exterior derivative between spaces of flat forms is dual
to the boundary operator between spaces of flat chains. We have 
\begin{align}
 \label{math:stokes}
 \int_{\partial\alpha} \omega
 =
 \int_{\alpha} \cartan\omega,
 \quad 
 \alpha \in \calC^{\flat}_{k}(\bbR^{n}),
 \quad 
 \omega \in L^{\infty,\infty}\Lambda^{k}({\bbR^n}),
\end{align}
as a generalized Stokes' theorem.

We consider pushforwards of chains and pullbacks of differential forms along Lipschitz mappings.
% Here we follow Paragraph 7 in Chapter X of \cite{whitney2012geometric}.
Assume that $\phi : {\bbR^m} \rightarrow {\bbR^n}$ is a Lipschitz mapping.
Then there exists a mapping $\phi_{\ast} : \calC^{\flat}_{k}({\bbR^m}) \rightarrow \calC^{\flat}_{k}({\bbR^n})$,
the \emph{pushforward} along $\phi$, such that
\begin{gather}
 \label{math:pushforward_commutes}
 \partial \phi_{\ast} \alpha
 =
 \phi_{\ast} \partial \alpha,
 \quad
 \alpha \in \calC^{\flat}_{k}({\bbR^m}),
 \\
 \label{math:pushforward_flatestimate}
 \| \phi_{\ast} \alpha \|_{k,\flat} 
 \leq
 \sup\{\Lip(\phi,\bbR^{m})^{k},\Lip(\phi,\bbR^{m})^{k+1}\} \| \alpha \|_{k,\flat},
 \quad
 \alpha \in \calC^{\flat}_{k}({\bbR^m}),
 \\
 \label{math:pushforward_massestimate}
 | \phi_{\ast} S |_{k} 
 \leq
 \Lip(\phi,\bbR^{m})^{k} | S |_{k},
 \quad
 S \in \overline{P^{k}({\bbR^m})}.
\end{gather}
% Dually, there exists a mapping $\phi^{\ast} : \calC_{\flat}^{k}({\bbR^m}) \rightarrow \calC_{\flat}^{k}({\bbR^n})$,
Dually, there exists a mapping $\phi^{\ast} : L^{\infty,\infty}\Lambda^{k}({\bbR^n}) \rightarrow L^{\infty,\infty}\Lambda^{k}({\bbR^m})$,
called the \emph{pullback} along $\phi$, which satisfies
\begin{gather}
 \label{math:pullback_commutes}
 \cartan \phi^{\ast} \omega
 =
 \phi^{\ast} \cartan \omega,
%  \quad
%  \omega \in L^{\infty,\infty}\Lambda^{k}({\bbR^m}),
 \\
 \label{math:pullback_flatestimate:itself}
 \| \phi_{\ast} \omega \|_{L^{\infty}\Lambda^{k}(\bbR^{n})}
 \leq
 \Lip(\phi,\bbR^{n})^{k} 
 \| \omega \|_{L^{\infty}\Lambda^{k}(\bbR^{n})},
 \\
 \label{math:pullback_flatestimate:derivative}
 \| \phi_{\ast} \cartan \omega \|_{L^{\infty}\Lambda^{k+1}(\bbR^{m})}
 \leq
 \Lip(\phi,\bbR^{n})^{k+1} 
 \| \cartan \omega \|_{L^{\infty}\Lambda^{k+1}(\bbR^{n})},
 % 
% \| \phi_{\ast} \omega \|_{L^{\infty,\infty}\Lambda^{k}(\bbR^{n})}
%  \leq
%  \binom{n}{k} \left( \Lip(\phi,\bbR^{n})^{k} + \Lip(\phi,\bbR^{n})^{k+1} \right)
%  \| \omega \|_{L^{\infty,\infty}\Lambda^{k}(\bbR^{m})},
%  \quad
%  \omega \in L^{\infty,\infty}\Lambda^{k}({\bbR^m})
%  .
 % 
\end{gather}
for $\omega \in L^{\infty,\infty}\Lambda^{k}({\bbR^n})$.
The pushforward and the pullback are related by the identity
\begin{align}
 \label{math:pushforward_pullback_duality}
 \int_{ \phi_\ast \alpha } \omega = \int_{\alpha} \phi^{\ast} \omega,
 \quad
 \omega \in L^{\infty,\infty}\Lambda^{k}({\bbR^n}),
 \quad
 \alpha \in \calC^{\flat}_{k}({\bbR^m})
 .
\end{align}

Having outlined basic concepts of geometric measure theory,
we provide a new result which makes these notions productive for finite element theory:
the degrees of freedom in finite element exterior calculus
are flat chains.

\begin{lemma}
 \label{prop:dof_are_chains}
 Let $F \subset {\bbR^n}$ be a closed oriented $m$-simplex and let $\eta \in C^{\infty}\Lambda^{m-k}(F)$.
 Then there exists a flat chain $\alpha_{\eta} \in C^{\flat}_{k}({\bbR^n})$ such that
 \begin{align}
  \label{prop:dof_are_chains:equality}
  \int_F \trace_F \omega \wedge \eta = \int_{\alpha_{\eta}} \omega,
  \quad
  \omega \in L^{\infty,\infty}\Lambda^{k}({\bbR^n}).
 \end{align}
 Moreover, $\alpha_{\eta} \in \overline{P^{k}({\bbR^n})}$
 and $\partial\alpha_{\eta} \in \overline{P^{k-1}({\bbR^n})}$.
\end{lemma}

\begin{proof}
 %%% If F is full dimensional
 We first assume that $\dim F = n$, and that $F$ is positively oriented.
 As is well-known, there exists $\star\eta \in C^{\infty}\Lambda^{k}(F)$
 such that 
 \begin{align*}
  \int_{F} \omega \wedge \eta
  =
  \int_{F} \langle \omega, \star \eta \rangle,
  \quad 
  \omega \in L^{\infty,\infty}\Lambda^{k}({\bbR^n}),
 \end{align*}
 and $\| \eta \|_{L^{1}\Lambda^{n-k}(F)} = \| \star\eta \|_{L^{1}\Lambda^{k}(F)}$.
%  Further, there exists a smooth $k$-vector field $\sharp\star\eta^{}$
%  in the sense of \cite[Chapter I]{whitney2012geometric} such that,
%  \begin{align*}
%   \int_{F} \langle \omega, \star \eta \rangle 
%   = 
%   \int_{F} \langle \omega, \sharp\star\eta \rangle, 
%   \quad 
%   \omega \in L^{\infty,\infty}\Lambda^{k}({\bbR^n})
%  \end{align*}
%  and $\| \star\eta \|_{L^{1}\Lambda^{k}(F)} = \| \sharp\star\eta \|_{L^{1}\Lambda^{k}(F)}$.
 We use Theorem 15A of \cite[Chapter IX]{whitney2012geometric} 
 to deduce the existence of $\alpha_{\eta} \in C^{\flat}_{k}({\bbR^n})$
 such that 
 \begin{align*}
  \int_{\alpha_{\eta}} \omega  
  = 
  \int_{F} \langle \omega, \iffalse \sharp \fi \star\eta \rangle, 
  \quad 
  \omega \in L^{\infty,\infty}\Lambda^{k}({\bbR^n}),
 \end{align*}
 and $|\alpha_{\eta}|_{k} = \| \iffalse \sharp \fi\star\eta \|_{L^{1}\Lambda^{k}(F)}$.
 In particular, $\alpha_{\eta} \in \overline{P^{k}({\bbR^n})}$.
 
 Now assume that $\dim F = m < n$. There exists a simplex $F_0 \subseteq \bbR^{m}$
 and an isometric inclusion $\phi : \bbR^{m} \rightarrow \bbR^{n}$
 which maps $F_0$ onto $F$. 
 Recall that the pullback of a flat form along a Lipschitz mapping is well-defined. 
%  By Theorem 2.2 of \cite{gol1982differential}, 
 We have 
 \begin{align*}
  \int_{F} \trace_{F} \omega \wedge \eta
  =
  \int_{\phi F_0} \trace_{F} \omega \wedge \eta
  =
  \int_{F_0} \phi^{\ast}\trace_{F} \omega \wedge \phi^{\ast}\eta
  =
  \int_{\alpha_{\phi^{\ast}\eta}} \phi^{\ast}\trace_{F} \omega
  =
  \int_{\phi_{\ast}\alpha_{\phi^{\ast}\eta}} \omega
 \end{align*}
 for $\omega \in L^{\infty,\infty}\Lambda^{k}({\bbR^n})$.
 Thus we may choose $\alpha_{\eta} = \phi_{\ast} \alpha_{\phi^{\ast}\eta} \in \overline{P^{k}({\bbR^n})}$.
 
 It remains to show that $\partial\alpha_{\eta} \in \overline{P^{k-1}({\bbR^n})}$.
 For $\omega \in L^{\infty,\infty}\Lambda^{k-1}({\bbR^n})$, we have
 \begin{align*}
  \int_{\partial\alpha_{\eta}} \omega 
  =
  \int_{\alpha_{\eta}} \cartan\omega 
  &=
  \int_{F} \eta \wedge \trace_F\cartan\omega 
  \\&=
  (-1)^{k+1}\int_{F} \cartan\eta \wedge \trace_F\omega
  +
  \sum_{ \substack{ G \in \calT^{m-1} \\ G \subset \partial F } }
  \int_{G} \trace_{G} \eta \wedge \trace_{G} \omega
  . 
 \end{align*}
 Here, the sum is taken over all faces of $F$ of dimension $m-1$,
 equipped with the outward orientation. 
 This is a sum of functionals as in the statement of the theorem.
 Since $\overline{P^{k-1}({\bbR^n})}$ is a Banach space,
 we have $\partial\alpha \in \overline{P^{k-1}({\bbR^n})}$.
 The proof is complete. 
\end{proof}

We finish this section with an estimate on the deformation of flat chains by Lipschitz mappings.
The following result is implied by Theorem 13A in Chapter~X % of \cite{whitney2012geometric}
and the discussions in Paragraphs 1 and~2 of Chapter~VIII in \cite{whitney2012geometric}.
% Our application is the proof of Theorem~\ref{prop:interpolation_error}.

\begin{lemma}
 \label{prop:deformationestimate}
 Let $F \subseteq \bbR^{n}$ be an $m$-simplex and let $\eta \in C^{\infty}\Lambda^{m-k}(F)$.
 Let $\alpha \in C^{\flat}_{k}(\bbR^{n})$ be the associated flat chain
 in the manner of Lemma~\ref{prop:dof_are_chains}.
 Let $U \subseteq \bbR^{n}$ be open and convex with $F \subset U$,
 and let $\phi : U \rightarrow \bbR^{n}$ be Lipschitz.
 Then
 \begin{align}
  \| \phi_{\ast} \alpha - \alpha \|_{k,\flat} 
  \leq 
  \| \phi - \Id \|_{L^{\infty}(U,\bbR^{n})}
  \left( l^{k} |\alpha|_{k} + l^{k-1} |\partial \alpha|_{k-1} \right),
 \end{align}
 where $l := \sup\{ \Lip(\phi,U), 1 \}$.
\end{lemma}

% V. Interpolator

\section{Finite Element Spaces, Degrees of Freedom, and Interpolation}
\label{sec:interpolator}

In this section we outline the discretization theory of finite element exterior calculus.
We summarize basic facts on the finite element spaces and their spaces of degrees of freedom.
The most important construction is the canonical finite element interpolator $I^{k}$.
Moreover we consider several inverse inequalities. 
The reader is assumed to be familiar with the background in \cite{afwgeodecomp} and \cite[Section 3--5]{AFW1}.
We outline this background and additionally apply geometric measure theory
in the perspective of the preceding section. 

For the duration of this section, 
we fix a bounded weakly Lipschitz domain $\Omega \subset \bbR^{n}$
and a finite triangulation $\calT$ of $\overline\Omega$. 
\\

The essential idea is to consider a differential complex of finite element spaces 
that mimics the de~Rham complex on a discrete level. 
The finite element spaces are finite-dimensional spaces of piecewise polynomial differential forms.

Let $T \in \calT^{n}$ be an $n$-simplex, and let $r,k \in \bbZ$.
We define $\calP_{r}\Lambda^{k}(T)$ as the space of differential $k$-forms
whose coefficients are polynomials over $T$ of degree at most $r$.
We define $\calP_{r}^{-}\Lambda^{k}(T) := \calP_{r-1}\Lambda^{k}(T) + \vecX \lrcorner \calP_{r-1}\Lambda^{k+1}(T)$,
where $\vecX\lrcorner$ denotes contraction with the source vector field $\vecX( x ) = x$. 
One can show that $\calP_{r}\Lambda^{k}(T)$ and $\calP_{r}^{-}\Lambda^{k}(T)$
are invariant under pullback by affine automorphisms of $T$. 
For any subsimplex $F \in \Delta(T)$ of $T$ we set 
\begin{align*}
 \calP_r\Lambda^{k}(F) = \trace_{T,F} \calP_r\Lambda^{k}(T),
 \quad 
 \calP^{-}_r\Lambda^{k}(F) = \trace_{T,F} \calP^{-}_r\Lambda^{k}(T),
\end{align*}
which do not depend on $T$. 
Some basic properties of these spaces are 
\begin{gather*}
 \calP_r\Lambda^{k}(T) \subseteq \calP_{r+1}^{-}\Lambda^{k}(T),
 \quad 
 \calP_r^{-}\Lambda^{k}(T) \subseteq \calP_{r}\Lambda^{k}(T),
 \\
 \cartan\calP_r\Lambda^{k}(T) \subseteq \calP_{r-1}\Lambda^{k+1}(T),
%  \quad
%  \cartan\calP_r^{-}\Lambda^{k}(T) \subseteq \calP_{r-1}\Lambda^{k+1}(T),
 \quad
 \cartan\calP_r\Lambda^{k}(T) = \cartan\calP_{r}^{-}\Lambda^{k}(T),
 \\ 
 \calP_r\Lambda^{0}(T) = \calP_{r}^{-}\Lambda^{0}(T),
 \quad 
 \calP_r\Lambda^{n}(T) = \calP_{r+1}^{-}\Lambda^{n}(T).
\end{gather*}
We define the finite element spaces 
\begin{align*}
 \calP_{r}\Lambda^{k}(\calT)
 &:=
 \left\{
 \omega \in L^{\infty,\infty}\Lambda^{k}(\Omega)
 \st 
 \forall T \in \calT^{n} : \omega_{|T} \in \calP_{r}\Lambda^{k}(T)
 \right\}
 ,
 \\
 \calP_{r}^{-}\Lambda^{k}(\calT)
 &:=
 \left\{
 \omega \in L^{\infty,\infty}\Lambda^{k}(\Omega)
 \st 
 \forall T \in \calT^{n} : \omega_{|T} \in \calP_{r}^{-}\Lambda^{k}(T)
 \right\}
 .
\end{align*}
These are spaces of piecewise polynomial differential forms.
The regularity $L^{\infty,\infty}\Lambda^{k}(\Omega)$
enforces that their members feature tangential continuity along simplex boundaries.
In particular, if $\omega \in \calP_r\Lambda^{k}(\calT)$
and $T,T' \in \calT$ are neighbouring simplices,
then the restrictions of $\omega$ to $T$ and $T'$
have the same trace on their common subsimplex $T \cap T'$.
We have recovered precisely the finite element spaces of \cite{AFW1}.

From $\calP_{r}\Lambda^{k}(\calT)$ and $\calP_{r}^{-}\Lambda^{k}(\calT)$
we can construct finite element de Rham complexes, 
but the combination of spaces is not arbitrary.
We single out a class of differential complexes 
that we call \emph{FEEC-complexes} in this article.
A \emph{FEEC-complex} is a differential complex 
\begin{align}
 \label{math:finiteelementcomplex}
 \begin{CD}
  0 \to
  \Lambda^{0}(\calT)
  @>\cartan>>
  \Lambda^{1}(\calT)
  @>\cartan>>
  \cdots
  @>\cartan>>
  \Lambda^{n}(\calT)
  \to 0
 \end{CD}
\end{align}
such that for all $k \in \bbZ$ there exists $r \in \bbZ$ with 
% provided that for all $k \in \bbZ$ there exists $r \in \bbZ$ such that 
\begin{align}
 \Lambda^{k  }(\calT) \in \left\{ \calP_{r  }\Lambda^{k  }(\calT), \calP_{r}^{-}\Lambda^{k  }(\calT) \right\}
\end{align}
and that for all $k \in \bbZ$ we have 
\begin{align}
 \begin{split}&
 \Lambda^{k  }(\calT) \in \left\{ \calP_{r  }\Lambda^{k  }(\calT), \calP_{r}^{-}\Lambda^{k  }(\calT) \right\}
 \\&\qquad
 \implies 
 \Lambda^{k+1}(\calT) \in \left\{ \calP_{r-1}\Lambda^{k+1}(\calT), \calP_{r}^{-}\Lambda^{k+1}(\calT) \right\}.
 \end{split}
\end{align}
These are the finite element de~Rham complexes discussed in \cite{AFW1}.

We next introduce the degrees of freedom of finite element exterior calculus.
They are represented by taking the trace of a differential form
onto a simplex of $\calT$ and then integrating against a smooth differential form.
By virtue of Lemma~\ref{prop:dof_are_chains},
we introduce the degrees of freedom as chains of finite mass.
% but take into account that these can interpreted as flat chains.
% 
% Indeed, the canonical degrees of freedom in finite element exterior calculus
% can be identified with polyhedral chains
% by virtue of Lemma \ref{prop:dof_are_chains}.
% 
Specifically, when $F \in \calT$ and $m = \dim(F)$, then we define 
\begin{align*}
 \calP_r\calC_k^{F}
 &:=
 \left\{
  S \in \calC_k^{\flat}(\bbR^{n}) 
  \;\bigg|\;
  \exists \eta_S \in \calP_{r+k-m}^{-}\Lambda^{m-k}(F) : 
  \int_S \cdot = \int_F \eta_S \wedge \cdot
 \right\},
 \\
 \calP_r^{-}\calC_k^{F}
 &:=
 \left\{
  S \in \calC_k^{\flat}(\bbR^{n}) 
  \;\bigg|\;
  \exists \eta_S \in \calP_{r+k-m-1}\Lambda^{m-k}(F) : 
  \int_S \cdot = \int_F \eta_S \wedge \cdot
 \right\}.
\end{align*}
We furthermore obtain by Lemma~\ref{prop:dof_are_chains}
that the degrees of freedom are flat chains
of finite mass with boundaries of finite mass.
One can show that we have direct sums  
\begin{align*}
 \calP_r\calC_k(\calT)
 :=
 \sum_{ F \in \calT } \calP_r\calC_k^{F},
 \quad 
 \calP_r^{-}\calC_k(\calT)
 :=
 \sum_{ F \in \calT } \calP_r^{-}\calC_k^{F}
\end{align*}
and that we have the inclusions 
% Moreover, these flat chains have boundaries of finite mass.
% One can show that  
\begin{align*}
 \partial \calP_r\calC_k(\calT)
 \subseteq 
 \calP_{r+1}^{-}\calC_{k-1}(\calT),
  \quad 
 \calP_r\calC_k(\calT)
 \subseteq 
 \calP^{-}_{r+1}\calC_k(\calT),
  \quad 
 \calP^{-}_{r}\calC_k(\calT)
 \subseteq 
 \calP_{r}\calC_k(\calT)
 . 
\end{align*}
With respect to a given FEEC-complex \eqref{math:finiteelementcomplex},
we then define 
% for some choice of finite element spaces. 
% With respect to this fixed complex, we define $\calC_{k}(\calT)$ by 
\begin{align}
  \calC_k(\calT) = \left\{
 \begin{array}{rl}
  \calP_r\calC_k(\calT) & \text{ if } \Lambda^{k}(\calT) = \calP_r\Lambda^{k}(\calT),
  \\
  \calP_r^{-}\calC_k(\calT) & \text{ if } \Lambda^{k}(\calT) = \calP_r^{-}\Lambda^{k}(\calT).
 \end{array}
 \right.
\end{align}
% \begin{align*}
%  \calC_k(\calT)
%  :=
%  \sum_{ F \in \calT } \calC_k^{F},
%  \quad 
%  \calC_k^{F}
%   := \left\{
%  \begin{array}{rl}
%   \calP_r\calC_k^{F} & \text{ if } \Lambda^{k}(\calT) = \calP_r(\calT),
%   \\
%   \calP_r^{-}\calC_k^{F} & \text{ if } \Lambda^{k}(\calT) = \calP_r^{-}(\calT),
%  \end{array}
%  \right.
% \end{align*}
for $k \in \bbZ$. 
Note that $\partial \calC_{k+1}(\calT) \subseteq \calC_{k}(\calT)$ by construction.
% , and furthermore 
% \begin{align}
%   \calC_k(\calT) = \left\{
%  \begin{array}{rl}
%   \calP_r\calC_k(\calT) & \text{ if } \Lambda^{k}(\calT) = \calP_r(\calT),
%   \\
%   \calP_r^{-}\calC_k(\calT) & \text{ if } \Lambda^{k}(\calT) = \calP_r^{-}(\calT).
%  \end{array}
%  \right.
% \end{align}
We have a well-defined complex of degrees of freedom
\begin{align}
 \label{math:degreesoffreedomcomplex}
 \begin{CD}
  0 \gets
  \calC_{0}(\calT)
  @<\partial<<
  \calC_{1}(\calT)
  @<\partial<<
  \cdots
  @<\partial<<
  \calC_{n}(\calT)
  \gets 0.
 \end{CD}
\end{align}
We can prove a duality between the finite element complex \eqref{math:finiteelementcomplex}
and the complex of degrees of freedom \eqref{math:degreesoffreedomcomplex},
following \cite[Section 5]{afwgeodecomp}.
One can show that
\begin{subequations}
\label{math:discreteduality}
\begin{align}
 \label{math:discreteduality:feelfemspace}
 \forall S \in \calC_k(\calT) : 
 S \neq 0 \implies 
 \exists \omega \in \Lambda^k(\calT) : 
 \int_S \omega \neq 0,
 \\
 \label{math:discreteduality:feeldofspace}
 \forall \omega \in \Lambda^k(\calT) : 
 \omega \neq 0 \implies 
 \exists S \in \calC_k(\calT) : 
 \int_S \omega \neq 0.
\end{align}
\end{subequations}
We conclude that $\calC_k(\calT)$,
restricted to $\Lambda^{k}(\calT)$, spans the dual space of $\Lambda^{k}(\calT)$.
Notably, the last implication can be strengthened to the following ``local'' result.
When $T \in \calT^{n}$ and $\omega \in \Lambda^{k}(\calT)$, then 
\begin{align}
 \label{math:dof:dualspace}
%  \forall T \in \calT^{n} : 
%  \forall \omega \in \Lambda^k(\calT) :
%  \left(
 \omega_{|T} = 0
 \quad\equivalent\quad 
 \forall F \in \Delta(T) : 
 \forall S \in \calC_k^{F} : 
 \int_S \omega = 0
 .
%  \right),
%  \\
%  \forall T \in \calT^{n} : 
%  \forall \omega \in \calP_r\Lambda^k(\calT) :
%  \left(
%  \omega_{|T} = 0
%  \quad\equivalent\quad 
%  \forall F \in \Delta(T) : 
%  \forall S \in \calP_r\calC_k^{F} : 
%  \int_S \omega = 0
%  \right),
%  \\
%  \forall T \in \calT^{n} : 
%  \forall \omega \in \calP_r\Lambda^k(\calT) :
%  \left(
%  \omega_{|T} = 0
%  \quad\equivalent\quad 
%  \forall F \in \Delta(T) : 
%  \forall S \in \calP_r\calC_k^{F} : 
%  \int_S \omega = 0
%  \right)
 %
\end{align}
So the value of $\omega \in \Lambda^{k}(\calT)$ is determined uniquely
by the values of the degrees of freedom associated with that simplex.
% TODO: Umgekehrt?

\begin{remark}
 At this point it is helpful to recall the role of degrees of freedom in finite element theory.
 On the one hand, 
 they are functionals which span the dual space of a finite element space,
 and on the other hand, 
 the degrees of freedom are used in the construction of the canonical finite element interpolator.
 % Hence the degrees of freedom are defined on differential forms with less than full continuity
 % along interelement boundaries, and on continuous differential forms.
 Corresponding to these two applications, we treat the degrees of freedom
 as functionals both over $L^{\infty,\infty}\Lambda^{k}(\Omega^{e})$
 and $C\Lambda^{k}(\Omega^{e})$,
 and we define the canonical finite element interpolator over both spaces.
 % As a guideline we propose that the degrees of freedom should be defined both on $C\Lambda^{k}(\Omega^{e})$
 % and $L^{\infty,\infty}\Lambda^{k}(\Omega^{e})$.
 This is possible because the degrees of freedom are flat chains of finite mass.
%  This is reflected below in our treatment of the canonical finite element interpolator.
\end{remark}

We introduce the \emph{canonical finite element interpolator}.
This linear mapping is well-defined and bounded both over $C\Lambda^{k}(\Omega^{e})$
and $L^{\infty,\infty}\Lambda^{k}(\Omega^{e})$.
We define  
\begin{align}
 \label{math:feinterpolator}
 I^{k} : 
 C\Lambda^{k}(\Omega^{e}) + L^{\infty,\infty}\Lambda^{k}(\Omega^{e})
 \rightarrow 
 \Lambda^{k}(\calT)
\end{align}
by requiring that 
\begin{align}
 \label{math:feinterpolator:definingcondition}
%  \forall  : 
 \int_S \omega = \int_S I^{k}\omega,
 \quad 
 S \in \calC_{k}(\calT),
 \quad 
 \omega \in C\Lambda^{k}(\Omega^{e}) + L^{\infty,\infty}\Lambda^{k}(\Omega^{e})
 .
\end{align}
The finite element interpolator commutes with the exterior derivative,
which follows easily from \eqref{math:feinterpolator:definingcondition} and \eqref{math:stokes}. 
We have 
\begin{align}
 \int_S I^{k+1} \cartan \omega 
 =
 \int_S \cartan \omega 
 =
 \int_{\partial S} \omega 
 =
 \int_{\partial S} I^{k} \omega 
 =
 \int_{S} \cartan I^{k} \omega 
\end{align}
% \begin{align}
%  S\left( I^{k+1} \cartan \omega \right)
%  =
%  S\left( \cartan \omega \right)
%  =
%  (\partial S)\left( \omega \right)
%  =
%  (\partial S)\left( I^{k}\omega \right)
%  =
%  S\left( \cartan I^{k+1} \omega \right)
% \end{align}
for all $\omega \in L^{\infty,\infty}\Lambda^{k}(\Omega^{e})$ and $S \in \calC_{k+1}(\calT)$.
In particular, the following diagram commutes:
\begin{align}
 \label{math:interpolationoperatorcommutingdiagram}
 \begin{CD}
  \cdots
  @>>>
  L^{\infty,\infty}\Lambda^{k}(\Omega^{e})
  @>{\cartan}>>
  L^{\infty,\infty}\Lambda^{k+1}(\Omega^{e})
  @>>>
  \cdots
  \\
  @.
  @V{I^{k}}VV
  @VI^{k+1}VV
  @.
  \\
  \cdots
  @>>>
  \Lambda^{k}(\calT) 
  @>{\cartan}>> 
  \Lambda^{k+1}(\calT) 
  @>>> 
  \cdots
 \end{CD}
\end{align}
Furthermore $I^{k}$ is idempotent, i.e.\ 
\begin{align}
 \label{math:feinterpolator:idempotent}
 I^{k} \omega = \omega,
 \quad 
 \omega \in \Lambda^{k}(\calT)
 ,
\end{align}
as follows directly from \eqref{math:dof:dualspace}.
% 
% Recall that finite element space is a subspace of the flat differential forms.
% By \eqref{math:dof:dualspace} we have 
% \begin{align}
%  \label{math:feinterpolator:idempotent}
%  I^{k} \omega = \omega,
%  \quad 
%  \omega \in \Lambda^{k}(\calT).
% \end{align}
% So $I^{k}$ is idempotent. 
\\

In the remainder of this section, we introduce a number of \emph{inverse inequalities}.
These rely on the equivalence of norms over finite-dimensional vector spaces.

We note that, by construction, the pullbacks $A_T^{\ast} \omega_{|T}$
lie in a common finite-dimensional vector space as $\omega \in \Lambda^{k}(T)$ and $T \in \calT^{n}$ vary.
For example, this can be a fixed space of differential forms with polynomial coefficients 
of sufficiently high degree.
Hence for each $p \in [1,\infty]$ there exists a constant $C_{\flat,p,n} > 0$ such that
\begin{align}
 \label{math:inverseinequality:fe}
 \| A_T^{\ast} \omega \|_{L^{\infty,\infty}\Lambda^{k}(\Delta^{n})}  
 \leq C_{\flat,p,n}
 \| A_T^{\ast} \omega \|_{L^{p}\Lambda^{k}(\Delta^{n})},  
 \quad 
 \omega \in \Lambda^{k}(\calT),
 \quad 
 T \in \calT^{n}.
\end{align}
The constant $C_{\flat,p,n}$ depends only on $n$ and the maximal polynomial degree
in the finite element de Rham complex.

Another inverse inequality applies to the degrees of freedom.
By Lemma~\ref{prop:dof_are_chains}, each degree of freedom can be identified with a flat chain of finite mass
whose boundary is again a flat chain of finite mass. 
In general, the boundary operator is an unbounded operator
as a mapping between spaces of polyhedral chains with respect to the mass norm.
But in the present setting, the pushforward of the degrees of freedom 
onto the reference simplex takes values in a finite-dimensional vector space. 
We conclude that there exists $C_{\partial} > 0$ such that 
\begin{align}
 \label{math:inverseinequality:dof}
 | A_{T\ast}^{\inv} \partial S |_{k-1}  
 \leq 
 C_{\partial} 
 | A_{T\ast}^{\inv} S |_{k},
 \quad 
 S \in \calC_k^{F},
 \quad
 F \in \Delta(T),
 \quad
 T \in \calT^{n}.
\end{align}
Again, the constant $C_{\partial}$ depends only on $n$ and the maximal polynomial degree
in the finite element de Rham complex.

Finally, we observe that there exists $C_I > 0$ such that 
\begin{align}
 \label{math:feminterpolator:micro}
 \| A_{T}^{\ast} I^{k} \omega \|_{L^{\infty}\Lambda^{k}(\Delta^{n})}
 \leq 
 C_I 
 \sup_{ \substack{ F \in \Delta(T) \\ S \in \calC^{F}_{k} } }
 |A_{T\ast} S|^{\inv}_{k}
 \int_{ A_{T\ast} S} A_{T}^{\ast} \omega
%  ,
%  \quad 
%  T \in \calT^{n},
%  \quad 
%  \omega \in C\Lambda^{k}(\Omega).
\end{align}
for all $T \in \calT^{n}$ and $\omega \in C\Lambda^{k}(\Omega)$.
Similar as before, the constant $C_{\partial}$ depends only on $n$ and the maximal polynomial degree
in the finite element de Rham complex.
Note that this inequality immediately implies 
\begin{align}
 \label{math:feminterpolator:simplexbound}
 \| A_{T}^{\ast} I^{k} \omega \|_{L^{\infty}\Lambda^{k}(\Delta^{n})}
 &\leq C_I 
 \| A_{T}^{\ast} \omega \|_{C\Lambda^{k}(\Delta^{n})},
 \quad 
 \omega \in C\Lambda^{k}(T).
\end{align}
% So the finite element interpolator satisfies a local bound on the reference simplex.
To see why \eqref{math:feminterpolator:micro} holds true, recall that 
$A_T^{\ast} I^{k} A_T^{-\ast}$ defines a linear mapping from $C\Lambda^{k}(A_T^{\inv}\Omega)$
onto a space of polynomial differential forms over the reference simplex $\Delta^{n}$.
By construction, 
\begin{align}
 \int_{ A_{T \ast}^{\inv} S } A_T^{\ast} I^{k} A_T^{-\ast} \omega 
 =
 \int_{ S } I^{k} A_T^{-\ast} \omega 
 =
 \int_{ S } A_T^{-\ast} \omega 
 =
 \int_{ A_T^{-\ast} S } \omega 
\end{align}
when $S \in \calC_{k}(\calT)$ and $A_T^{-\ast}\omega \in C\Lambda^{k}(\Omega)$.
Since the pushforwards of degrees of freedom and the pullbacks of finite element differential forms
to the reference simplex vary within finite dimensional vector spaces,
the existence of $C_I > 0$ follows. 

\begin{remark}
 The existence of constants $C_{\flat,p,n}$, $C_{\partial}$, and $C_I$ as above follows trivially
 if the triangulation $\calT$ and the sequences \eqref{math:finiteelementcomplex} and \eqref{math:degreesoffreedomcomplex}
 are fixed.
 But in applications we consider families of triangulations with associated   
 sequences \eqref{math:finiteelementcomplex} and \eqref{math:degreesoffreedomcomplex},
 and demand uniform bounds for those constants.
 Such uniform bounds hold if the triangulations have uniformly bounded shape constants
 and the finite element spaces have uniformly bounded polynomial degree. 
 The results of this article do not attend to estimates that are uniform
 in the polynomial degree, as would be relevant for $p$- and $hp$-methods.
\end{remark}

% If the triangulation $\calT$ and the sequences \eqref{math:finiteelementcomplex} and \eqref{math:degreesoffreedomcomplex}
% are fixed, then the existence of constants $C_{\flat,p,n}$, $C_{\partial}$, and $C_I$ 
% in the above inequalities follows trivially.
% But in applications, we are interested in uniform bounds for those constants
% when considering families of triangulations or 
% families of sequences of the form \eqref{math:finiteelementcomplex} and \eqref{math:degreesoffreedomcomplex}.
% Such uniform bounds hold true 
% if the triangulations have uniformly bounded shape constants
% and the finite element spaces have uniformly bounded polynomial degree. 

% % % VI. Regularizer
% % \input{regularizer_wlpub}
% VI. Smoothed Projection

\section{Smoothed Projection}
\label{sec:projection}

In this section,
we construct the smoothed projection in several stages.
First, we devise an extension operator $E^{k}$,
applying the two-sided Lipschitz collar discussed in Section~\ref{sec:geometry}.
We then formulate a mollification operator $R^{k}_{\eps \mathtt h}$,
where we use a smooth mesh size function $\mathtt h$ as an auxiliary construction.
Successive composition with the canonical finite element interpolator $I^{k}$
from Section~\ref{sec:interpolator}
yields an uniformly bounded commuting mapping $Q^{k}_{\eps}$,
the \emph{smoothed interpolator},
from differential forms with coefficients in $L^{p}$
onto finite element differential forms.
$Q^{k}_{\eps}$ is generally not idempotent on the finite element space,
but the interpolation error can be controlled. 
% In order to make this operation a projection,
% which acts as the identity on the finite element space,
% we require an estimate of the interpolation error
After a small modification, 
we obtain the desired smoothed projection $\pi^{k}_{\eps}$.
\\

Throughout this section we assume that $\Omega \subseteq \bbR^{n}$ 
is a bounded simply-connected weakly Lipschitz domain
and that $\calT$ is a finite triangulation of $\overline\Omega$.
We additionally assume that we have fixed 
a FEEC-complex \eqref{math:finiteelementcomplex}
% FIXME: Determine what we assume additionally 
and a corresponding complex of degrees of freedom \eqref{math:degreesoffreedomcomplex}.
In the sequel, we adhere to the convention of stating each result 
accompanied by explicit estimates of the various constants and parameter ranges.
We call a quantity \emph{uniformly bounded}
if it can be bounded in terms of the shape-constant, 
the geometric ambient, and the polynomial degree of the finite element space.

\subsection{Extension}
\label{subsec:extension}

Since $\Omega$ is a bounded weakly Lipschitz domain,
we may apply Theorem~\ref{prop:closedtwosidedcollar} 
to fix a compact neighborhood $\calC\Omega$ of $\partial\Omega$ in $\bbR^n$
and a bi-Lipschitz mapping
\begin{align*}
 \Psi : \partial\Omega \times [-1,1] \rightarrow \calC\Omega
\end{align*}
such that $\Psi(x,0) = x$ for $x \in \partial\Omega$,
and such that 
\begin{gather*}
 \Psi\left( \partial\Omega \times [-1,0) \right) = \calC\Omega \cap \Omega,
 \quad
 \Psi\left( \partial\Omega \times (0,1] \right) = \calC\Omega \cap \overline\Omega^c.
\end{gather*}
Additionally we write
\begin{align}
 \label{math:definitionofcollarparts}
 \calC^-\Omega := \calC\Omega \cap \Omega,
 \quad
 \calC^+\Omega := \calC\Omega \cap \overline\Omega^c,
 \quad
 \Omega^{e} := \overline\Omega \cup \calC^+\Omega 
\end{align}
for the \emph{interior collar part} $\calC^-\Omega$,
the \emph{exterior collar part} $\calC^+\Omega$,
and the \emph{extended domain} $\Omega^{e}$, respectively.
Eventually, we have a well-defined bi-Lipschitz mapping
\begin{align}
 \label{math:definitioncoordinatereflection}
 \calA : \calC^{+}\Omega \rightarrow \calC^{-}\Omega,
 \quad
 \Psi(x,t) \mapsto \Psi(x,-t)
\end{align}
from the outer collar part into the inner collar part,
called \emph{collar reflection}.
\\

We define the extension operator using the pullback along 
the collar reflection. % $\calA : \calC^{+}\Omega \rightarrow \calC^{-}\Omega$.
% which we have defined previously in Section~\ref{sec:geometry}.
If $\omega \in M\Lambda^{k}(\Omega)$ is a locally integrable $k$-form over $\Omega$,
then 
\begin{align}
 \label{math:definitionextensionoperator}
 E^{k} \omega 
 := \left\{
 \begin{array}{rl}
%   \omega(x) 
%   & \text{ if } x \in \Omega,
%   \\
%   \calA^{\ast} \omega(x)
%   & \text{ if } x \in \calC^{+}\Omega,
  \omega  
  & \text{ over } \Omega,
  \\
  \calA^{\ast} \omega
  & \text{ over } \calC^{+}\Omega,
 \end{array}\right.
\end{align}
is the locally integrable differential $k$-form constructed by extending $\omega$
onto $\calC^{+}\Omega$ using the pullback along $\calA$.
We first show that the linear mapping $E^{k}$
satisfies local estimates:

\begin{lemma}
 \label{prop:extensionestimate}
 Let $p \in [1,\infty]$. 
 We have a bounded linear operator
 \begin{align*}
  E^{k} : L^{p}\Lambda^{k}(\Omega) \rightarrow L^{p}\Lambda^{k}(\Omega^{e}),
  \quad
  \omega \mapsto E^{k} \omega
  .
 \end{align*}
 Moreover, there exists $C_{\calA,p} > 0$,
 depending only on $\calA$ and $p$, such that
 for $0 \leq s \leq t \leq 1$ and ${G} \subseteq \partial\Omega$ closed 
 we have 
 \begin{align}
  \label{math:extensionoperator:localbound}
  \| E^{k} \omega \|_{ L^{p}\Lambda^{k}\Psi({G}, [s,t] ) }
  \leq
  C_{\calA,p}
  \| \omega \|_{ L^{p}\Lambda^{k}\Psi({G}, [-t,-s] ) },
  \quad 
  \omega \in L^{p}\Lambda^{k}(\Omega).
 \end{align}
%  for $0 \leq s \leq t \leq 1$ and ${G} \subseteq \partial\Omega$ closed.
%  
 \textit{Constants: we may assume that $C_{\calA,q} \leq C_{\calA,p}$ for $1 \leq p \leq q \leq \infty$.}
\end{lemma}

\begin{proof}
 Let $p \in [1,\infty]$,
 let ${G} \subseteq \partial\Omega$ be closed,
 let $0 \leq s \leq t \leq 1$,
 and let $\omega \in L^{p}\Lambda^{k}(\Omega)$.
 We apply Lemma~\ref{prop:pullbackestimate} to find 
 \begin{align*}
  &\| E^{k} \omega \|_{ L^{p}\Lambda^{k}\Psi({G}, [s,t] ) }
  =
  \| \calA^{\ast} \omega \|_{ L^{p}\Lambda^{k}\Psi({G}, [s,t] ) }
  \\&\quad\quad\leq
%   \binom{n}{k}^{\frac{1}{p}} (n!)^{\frac{1}{p}}
  \| \Jacobian \calA \|^{k}_{L^{\infty}(\calC^{+}\Omega)}
  \| \Jacobian \calA^{\inv} \|^{\frac{n}{p}}_{L^{\infty}(\calC^{-}\Omega)}
  \| \omega \|_{ L^{p}\Lambda^{k}\Psi({G}, [-t,-s] ) }
  .
 \end{align*}
 Hence \eqref{math:extensionoperator:localbound} 
 holds for some $C_{\calA,p} > 0$.
 For ${G} \times [s,t] = \partial\Omega \times [0,1]$
 we find  
 \begin{align*}
  \| E^{k} \omega \|_{L^{p}\Lambda^{k}(\Omega^{e})}
  &\leq 
  \| \omega \|_{L^{p}\Lambda^{k}(\Omega)}
  +
  \| E^{k} \omega \|_{L^{p}\Lambda^{k}(\calC^{+}\Omega)}
  \\&\leq 
  \| \omega \|_{L^{p}\Lambda^{k}(\Omega)}
  +
  C_{\calA,p} 
  \| \omega \|_{L^{p}\Lambda^{k}(\calC^{-}\Omega)}
  \\&\leq 
  \left( 1 + C_{\calA,p} \right) 
  \| \omega \|_{L^{p}\Lambda^{k}(\Omega)}
  , 
 \end{align*}
 so $E^{k}$ is bounded 
 from $L^{p}\Lambda^{k}(\Omega)$ to $L^{p}\Lambda^{k}(\Omega^{e})$.
\end{proof}

The local bound in the preceding lemma can be refined:

\begin{lemma}
 \label{prop:extensionestimate:refined}
 There exist 
 $\delta_0 > 0$
 and
 $L_{\Psi} \geq 1$, depending only on $\Psi$,
 such that 
 for all $\delta \in [0, \delta_0 )$, $p \in [1,\infty]$,
 and all closed sets $A \subset \overline\Omega$
 we have 
 \begin{align}
  \| E^{k} \omega \|
  _{L^p\Lambda^{k}\left( B_{\delta}(A) \cap \Omega^{e} \right) }
  \leq 
  \left( 1 + C_{\calA,p} \right)
  \| \omega \|
  _{L^p\Lambda^{k}\left( B_{L_{\Psi} \delta}(A) \cap \overline\Omega \right) },
  \quad 
  \omega \in L^{p}\Lambda^{k}(\Omega)
  .
 \end{align}
%  for $A \subseteq \overline\Omega$ closed and $p \in [1,\infty]$.
\end{lemma}

\begin{proof}
 Let $\delta \geq 0$, $p \in [1,\infty]$, and let $A \subset \overline\Omega$ be closed.
 Then 
 \begin{align*}
  \| E^{k} \omega \|_{L^p\Lambda^{k}\left( B_{\delta}(A) \cap \Omega^{e} \right) }
  &\leq 
  \| \omega \|_{L^p\Lambda^{k}\left( B_{\delta}(A) \cap \Omega \right) }
  +
  \| E^{k} \omega \|_{L^p\Lambda^{k}\left( B_{\delta}(A) \cap \calC^{+}\Omega \right) }
  .
 \end{align*}
 We set $H^{+} := B_{\delta}(A) \cap \calC^{+}\Omega$.
 There exists $H^{-} \subseteq \overline\Omega$ such that $H^{-} = \calA( H^{+} )$.
 By the definition of $E^{k}$ and Lemma~\ref{prop:pullbackestimate} we have  
 \begin{gather*}
  \| E^{k} \omega \|_{L^p\Lambda^{k}\left( H^{+} \right) }
  \leq  
  \| \Jacobian \calA \|^{k}_{L^{\infty}(\calC^{+}\Omega)}
  \| \Jacobian \calA^{\inv} \|^{\frac{n}{p}}_{L^{\infty}(\calC^{-}\Omega)}
  \| \omega \|_{ L^{p}\Lambda^{k}(H^{-}) }
  .  
 \end{gather*}
 Let $x \in B_{\delta}(A) \cap \calC^{+}\Omega$.
 There exists $z \in B_{\delta}(A) \cap \partial\Omega$
 with $\| x - z \| \leq \delta$,
 since every $x \in A$ and $z \in \calC^{+}\Omega$
 with $\| x - z \| \leq \delta$ are connected by a straight line segment of length at most $\delta$
 which intersects $\partial\Omega$ at least once.
 Furthermore
 there exist $x_0 \in \partial\Omega$ and $t \in [0,1]$
 with $x = \Phi(x_0,t)$.
 It is easily seen that $\| \calA(x) - x \| \leq C_1 t$
 for some constant $C_1$ that depends only on $\Psi$.
 Since $\Psi$ is a LIP embedding, we also know that 
 $\| x_0 - z \|^{2} + | t |^{2} \leq C_2 \| x - z \|$
 for some constant $C_2$ that depends only on $\Psi$.
 In combination we have 
%  \begin{gather*}
 $H^{-} \subseteq B_{(1 + C_1 C_2 )\delta}(A)$.
%  \end{gather*}
 This completes the proof.
\end{proof}

\begin{corollary}
 \label{prop:extensionestimate:furtherrefined}
 Let $p \in [1,\infty]$ and let $\eps > 0$ be small enough.
 There exists $C_{E,p} > 0$ such that
 for $\omega \in L^{p}\Lambda^{k}(\Omega)$
 and $F \in \calT$ we have 
 \begin{align*}
  \| E^{k} \omega \|
  _{L^p\Lambda^{k}\left( B_{\eps h_F}(F) \cap \Omega^{e} \right) }
  \leq 
  C_{E,p}
  \| \omega \|
  _{L^p\Lambda^{k}\left( B_{L_{\Psi} \eps h_F}(F) \cap \overline\Omega \right) }
  .
 \end{align*}
 \textit{Constants: we may assume $C_{E,p} := \left( 1 + C_{\calA,p} \right)$.
 It suffices that $\eps h_F < \delta_0$ for all $F \in \calT$.}
\end{corollary}

Next, we show that $E^{k}$ commutes with the exterior derivative.
% Thus, if a differential form has a well-defined exterior derivative,
% then so does its extension, and the extension of the exterior derivative
% is the exterior derivative of the extension.

\begin{lemma}
 \label{prop:extensionderivativecommute}
 Let $p, q \in [1,\infty]$. 
 If $\omega \in L^{p,q}\Lambda^{k}(\Omega)$,
 then $E^{k} \omega \in L^{p,q}\Lambda^{k}(\Omega^{e})$
 and $E^{k} \cartan \omega = \cartan E^{k} \omega$.
\end{lemma}

\begin{proof}
 It suffices to consider the case $p = q = 1$.
 Let $\omega \in L^{1,1}\Lambda^{k}(\Omega)$.
 Lemma~\ref{prop:extensionestimate} implies that 
 $E^{k} \omega \in L^{1}\Lambda^{k}(\Omega^{e})$
 and $E^{k} \cartan \omega \in L^{1}\Lambda^{k+1}(\Omega^{e})$.
 To show that $E^{k} \omega \in L^{1,1}\Lambda^{k}(\Omega^{e})$,
 it suffices to show that there exists a covering $(U_{i})_{i \in \bbN}$ of $\Omega^{e}$
 by relatively open subsets $U_i \subseteq \Omega^{e}$
 such that $E^{k} \omega_{|U_i} \in L^{1,1}\Lambda^{k}(U_i)$
 and $E^{k} \cartan\omega_{|U_i} = \cartan E^{k} \omega_{|U_i}$.
 
 From the definition of weakly Lipschitz domains 
 we easily see that there exists a family 
 $(\theta_i)_{i \in \bbN}$ of LIP embeddings 
 $\theta_{i} : (-1,1)^{n-1} \rightarrow \partial\Omega$
 whose images cover $\partial\Omega$.
 Consequently, the mappings 
 $\phi_{i} : (-1,1)^{n} \rightarrow \Psi(\partial\Omega,(-1,1))$
 defined by $\phi_{i}( \theta(x), t ) = \Psi(x,t)$
 are a finite family of LIP embeddings
 whose images $U_i := \phi_{i}\left( (-1,1)^{n} \right)$ cover $\calC\Omega$.
 Together with $\Omega$ we thus have a finite covering of $\Omega^{e}$.
 
 We recall that $E^{k} \omega_{|\Omega} \in L^{1,1}\Lambda^{k}(\Omega)$
 and $E^{k} \cartan\omega_{|\Omega} = \cartan E^{k} \omega_{|\Omega}$.
 Next we define 
 \begin{gather*}
  \omega_i := \phi_{i}^{\ast}\left( E^{k} \omega_{|U_i}\right),
  \quad 
  \xi_i := \phi_{i}^{\ast}\left( E^{k+1} \cartan \omega_{|U_i}\right).
 \end{gather*}
 It remains to show that $E^{k} \omega_{|U_i} \in L^{1,1}\Lambda^{k}(U_i)$.
 and $E^{k} \cartan\omega_{|U_i} = \cartan E^{k} \omega_{|U_i}$,
 which is equivalent to $\omega_{i} \in L^{1,1}\Lambda^{k}((-1,1)^{n})$
 and $\cartan \omega_i = \xi_i$.
 To see this, we let 
 \begin{gather*}
  \scrR : (-1,1)^{n-1} \times (0,1) \rightarrow (-1,1)^{n-1} \times (-1,0)
 \end{gather*}
 be the reflection by the $n$-th coordinate.
 It is evident that 
 \begin{gather*}
  \omega_{i|(-1,1)^{n-1} \times (0,1)} 
  =
  \scrR^{\ast} \omega_{i|(-1,1)^{n-1} \times (-1,0)} 
  \\ 
  \xi_{i|(-1,1)^{n-1} \times (0,1)}
  =
  \scrR^{\ast} \xi_{i|(-1,1)^{n-1} \times (-1,0)}
  =
  \scrR^{\ast} \cartan \omega_{i|(-1,1)^{n-1} \times (-1,0)} 
 \end{gather*}
 By Lemma~\eqref{prop:approximation:Lpq}
 there exists a sequence $(\omega_i^{\delta})_{\delta > 0}$ of smooth differential $k$-forms
 that converge to $\omega_{i}$ over $(-1,1)^{n-1} \times (-1,0)$ in the $L^{1,1}$ norm
 for $\delta \rightarrow 0$. We let each $\omega_i^{\delta}$ be extended to 
 $(-1,1)^{n-1} \times (0,1)$ by pullback along $\scrR$.
 With this extension,
 $\omega_{i}^{\delta}$ converges to $\omega_{i}^{\delta}$ in $L^{1}\Lambda^{k}\left( (-1,1)^{n} \right)$
 and 
 $\cartan\omega_{i}^{\delta}$ converges to $\xi_{i}$ in $L^{1}\Lambda^{k}\left( (-1,1)^{n} \right)$
 for $\delta \rightarrow 0$.
 Hence $\omega_{i} \in L^{1}\Lambda^{k}\left( (-1,1)^{n} \right)$
 with $\cartan \omega_i = \xi_i$.
 
 The proof is complete.
\end{proof}

\subsection{Mesh size functions and Mollification}

The next step is constructing a commuting mollification operator.
We let the mollification radius vary over the domain,
so the operator satisfies local estimates uniformly 
for shape-regular families of triangulations. 
% We achieve this by letting the mollification radius vary over the domain 
% proportionally to the local mesh size.
A key component is a smooth function that indicates the local mesh size.

Recall the standard mollifier. This is a smooth function
\begin{align}
 \label{math:standardmollifier}
 \mu : \bbR^{n} \rightarrow [0,1],
 \quad
 y
 \mapsto \left\{ \begin{array}{rcl}
  C_{\mu} \exp\left( - ( 1 - |y|^{2})^{\inv} \right)
  & \text{ if } & |y| \leq 1, 
  \\
  0
  & \text{ if } & |y| > 1,
 \end{array}\right.
\end{align}
with compact support,
where $C_{\mu} > 0$ is chosen such that $\mu$ has unit integral.
We set $\mu_{r}(y) := r^{-n} \mu(y/r)$ for $y \in \bbR^{n}$ and $r > 0$.

First we prove the existence of a mesh size function $\ttH$ with Lipschitz regularity,
and then the existence of a mesh size function $\mathtt h$ that is smooth.
% We first show the existence of a mesh size function with Lipschitz regularity:

\begin{lemma}
 \label{prop:meshsizefunction}
 There exists
 $L_{\Omega} > 0$, only depending on $\Omega$, 
 and
 a Lipschitz continuous function $\ttH : \overline\Omega \rightarrow \bbR_0^{+}$
 such that
 \begin{gather}
  \label{math:meshfunctionrough:localcomparison}
  \forall F \in \calT :
  C_{mesh}^{\inv} h_F \leq \ttH_{|F} \leq C_{mesh} h_F,
  \\
  \label{math:meshfunctionrough:lipschitz}
  \Lip( \ttH, \overline\Omega ) 
  \leq
  C_{mesh} L_{\Omega}.
 \end{gather}
\end{lemma}

\begin{proof}
 Let the function $\ttH : \overline\Omega \rightarrow \bbR_0^{+}$ 
 be defined as follows. If $V \in \calT^{0}$, then we set $\ttH(V) = h_V$.
 We then extend $\ttH$ to each $T \in \calT$ by affine interpolation
 between the vertices of $T$.
 With this definition, $\ttH$ is continuous, and \eqref{math:meshfunctionrough:localcomparison}
 follows from \eqref{math:shapeconditions:comparison}.
 It remains to prove \eqref{math:meshfunctionrough:lipschitz}.
 Obviously, $\Lip( \ttH, T ) \leq C_{mesh}$ for $T \in \calT^{n}$.
 
 Since $\Omega$ is a bounded weakly Lipschitz domain, 
 there exists be a finite family $(U_i)_{1 \leq i \leq N}$
 of relatively open sets $U_i \subseteq \overline\Omega$
 such that 
 such that the union of all $U_i$ equals $\overline\Omega$,
 and such that there exist $\phi_i : \overline{U_i} \rightarrow [-1,1]^{n}$ bi-Lipschitz
 for each $1 \leq i \leq N$.
 By Lebesgue's number lemma\iffalse \cite[A13.4]{querenburg1973mengentheoretische} \fi,
 we may pick $\gamma > 0$ so small
 that for each $x \in \overline\Omega$
 there exists $1 \leq i \leq N$ such that 
 $B_{\gamma}(x) \cap \overline\Omega \subseteq U_i$.
 
 First assume that $x, y \in \Omega$ with $0 < \| x - y \| \leq \gamma$.
 Then there exists $1 \leq i \leq N$ with $x, y \in U_i$.
%  Let $S \subseteq [-1,1]^{n}$ be the line segment between $\phi(x)$ and $\phi(y)$.
 For $M \in \bbN$, consider a partition of
 the line segment in $[-1,1]^{n}$ from $\phi(x)$ to $\phi(y)$ 
 into $M$ subsegments
 of equal length with points $\phi_i(x)=z_0, z_1, \dots, z_M=\phi_i(x)$.
 Let $x_m := \phi_i^{\inv}(z_m) \in U_i$. 
 For $M$ large enough, the straight line segment between $x_{m-1}$ and $x_{m}$
 is contained in $U_i$ for all $1 \leq m \leq M$.
 After a further subpartitioning, not necessarily equidistant,
 we may assume to have a sequence $x = w_0, \dots, w_{M'} = y$
 for some $M' \in \bbN$ such that for all $1 \leq m \leq M'$
 the points $w_{m-1}$ and $w_{m}$ are connected by a straight line segment
 in $U_i$ and such that there exists $F_{m} \in \calT$ with $w_{m-1}, w_{m} \in F_{m}$.
 We observe  
 \begin{align*}
  | \ttH(y) - \ttH(x) |
  &\leq 
  \sum_{m=1}^{M'} | \ttH(w_{m}) - \ttH(w_{m-1}) |
  \\&\leq 
  C_{mesh} \sum_{m=1}^{M'} \| w_{m} - w_{m-1} \|
  \\&= 
  C_{mesh} \sum_{m=1}^{M} \| x_{m} - x_{m-1} \|
  \\&\leq 
  C_{mesh} \Lip(\phi_i^{\inv}) \sum_{m=1}^{M} \| \phi_i(x_{m}) - \phi_i(x_{m-1}) \|
  \\&\leq 
  C_{mesh} \Lip(\phi_i^{\inv}) \cdot \| \phi_i(y) - \phi_i(x) \|
  \\&\leq 
  C_{mesh} \Lip(\phi_i^{\inv}) \Lip(\phi_i) \cdot \| y - x \|
  . 
 \end{align*}
 If we instead assume that $x, y \in \Omega$ with $\| x - y \| \geq \gamma$, then 
 \begin{align*}
  | \ttH(y) - \ttH(x) |
  \leq
  \gamma^{\inv} \diam(\Omega) \cdot | \ttH(x) - \ttH(y) |
  \leq 
  \gamma^{\inv} \diam(\Omega) \cdot C_{mesh} \cdot \| y - x \|
  ,
 \end{align*}
 since $\gamma < \diam(\Omega)$.
 Hence $\Lip( \ttH, \Omega ) \leq C_{mesh} L_{\Omega}$ with 
 \begin{align*}
  L_{\Omega}
  :=
  \sup\left\{
   \gamma^{\inv} \diam(\Omega), 
   \; \Lip(\phi_1^{\inv}) \Lip(\phi_1),
   \dots, 
   \; \Lip(\phi_N^{\inv}) \Lip(\phi_N)
  \right\}.
 \end{align*}
 Thus $\Lip( \ttH, \overline\Omega ) \leq C_{mesh} L_{\Omega}$ 
 because any Lipschitz continuous function 
 is Lipschitz continuous over the closure of its domain with the same Lipschitz constant.
\end{proof}

\begin{remark}
 The preceding result was used before in literature,
 but estimating $\Lip(\ttH)$ did not receive much attention.
 An interesting observation is that $\Lip(\ttH)$
 is the product of $C_{mesh}$, which only depends on the shape of the simplices,
 and $L_{\Omega}$, which depends only the geometry. 
 Conceptually, $L_{\Omega}$ compares
 the \emph{inner path metric} of $\overline\Omega$
 to the Euclidean metric over $\overline\Omega$.
 The equivalence of these two metrics is non-trivial in general,
 but holds true for bounded weakly Lipschitz domains. 
\end{remark}

\begin{lemma}
 \label{prop:smoothmeshsizefunction}
 There exist a smooth function $\mathtt h : \Omega^{e} \rightarrow \bbR_0^{+}$
 and uniformly bounded constants $C_h > 0$ and $L_h > 0$
 such that
 \begin{gather}
  \label{math:meshfunction:localcomparison}
  \forall F \in \calT :
  \forall x \in F : 
  C_{h}^{\inv} h_F \leq \mathtt h(x) \leq C_{h} h_F,
  \\
  \label{math:meshfunction:lipschitz}
  \Lip( \mathtt h, \overline\Omega ) 
  \leq
  L_{h}.
 \end{gather}
%  for some geometry-dependent shape-uniform $L_{h} > 0$.
%  
 \textit{Constants: we may choose $C_h = C_{mesh}^{2}$
 and $L_{h} = \left( 1 + \Lip(\calA) \right) C_{mesh} L_{\Omega}$.
%  for some geometry-dependent $L_{\Omega} \geq 1$.
 }
\end{lemma}

\begin{proof}
 Let $\ttH : \overline\Omega \rightarrow \bbR^{+}_{0}$ be as in Lemma~\ref{prop:meshsizefunction}.
 We observe that $E^{0} \ttH$ is Lipschitz with 
 \begin{align*}
  \Lip\left( E^{0} \ttH, \Omega^{e} \right)
  \leq 
  \left( 1 + \Lip(\calA) \right) \Lip\left( \ttH, \overline\Omega \right).
 \end{align*}
 % This follows easily by a case distinction.
 Let $r > 0$ and define $\mathtt h := \mu_{r} \ast E^{0} \ttH$ 
 be the convolution of $E^{0} \ttH$ with the scaled mollifier $\mu_r$.
 For $r$ small enough it is easily verified that 
 \begin{align*}
  \Lip\left( \mathtt h, \overline\Omega \right)
  \leq 
  \Lip\left( E^{0} \ttH, \Omega^{e} \right).
 \end{align*}
 By standard results, $\mathtt h$ is smooth. Moreover we have 
 \begin{align*}
  \mathtt h(x)
  &=
  \int_{B_r(x) \cap \Omega} \mu_{r}(y) E^{0} \ttH(x + y) \;\dif y
  +
  \int_{B_r(x) \cap \calC^{+}\Omega} \mu_{r}(y) E^{0} \ttH(x + y) \;\dif y
%   \\&=
%   \int_{B_r(x) \cap \Omega} \mu_{r}(y) E^{0} \ttH(x + y) \;\dif y
%   +
%   \int_{\calA(B_r(x) \cap \overline\Omega^{c})} 
%   \mu_{r}\left( \calA^{\inv}(y) \right) \ttH( x + \calA^{\inv}(y) ) \det\Jacobian\calA^{\inv}\;\dif y
  .
 \end{align*}
 If $z \in B_r(x) \cap \overline\Omega^{c}$, 
 then $\calA(z) \in \Omega$ has at most distance $r + \Lip(\calA) r$ from $x$.
 We conclude that $\mathtt h(x)$ lies in the convex hull of values of $\ttH$
 over $B_{ r + \Lip(\calA) r }(x) \cap \overline\Omega$.
 Let $r > 0$ so small that $r + \Lip(\calA) r < h_{min} \eps_h$,
 where $h_{min}$ is the minimal diameter of the simplices in $\calT$
 and $\eps_h$ is as in \eqref{math:neighborcontainment}.
 Then $B_{ r + \Lip(\calA) r }(x) \cap \overline\Omega \subseteq \calT(T)$,
 and the desired statement follows.
\end{proof}

We use the mesh size function to define a family of LIP embeddings of $\Omega$ into $\Omega^{e}$.
For $\eps > 0$ we introduce  
\begin{align}
 \label{math:mollification_transform}
 \Phi_{\eps{\mathtt h}} :
 \overline\Omega \times \bbR^{n}
 \rightarrow
 \bbR^n,
 \quad
 \Phi_{\eps{\mathtt h}}( x, y )
 =
 x + \eps {\mathtt h}(x) y
 .
\end{align}
We abbreviate $\Phi_{\eps {\mathtt h},y}(x) := \Phi_{\eps {\mathtt h}}( x, y )$. 
Note that $\Phi_{\eps {\mathtt h}}$ is smooth:
\begin{align}
 \Jacobian_x \Phi_{\eps{\mathtt h}}( x, y )
 =
 \mathrm{Id} + \epsilon \cdot y \otimes \cartan {\mathtt h}_{|x}. 
\end{align}
% From this we immediately see that 
% In particular, 
It is easy to see that for $\eps > 0$ small enough,
$\Phi_{\eps\mathtt h,y}$ is a LIP embedding
with $\Phi_{\eps {\mathtt h}}(\Omega,B_1) \subseteq \Omega^{e}$
for any $y \in B_1(0)$. 
We now define the mollification operator $R^{k}_{\eps {\mathtt h}}$.
For $\omega \in L^1\Lambda^{k}(\Omega^{e})$
we let 
% the mollified differential form $R^{k}_{\eps {\mathtt h}} \omega$ by 
\begin{align}
 \label{math:mollification_operator}
 R^{k}_{\eps {\mathtt h}} \omega_{|x}
 &:=
 \int_{\bbR^n} \mu( y ) (\Phi_{\eps {\mathtt h},y}^{\ast} \omega)_{|x} \dif y,
 \quad
 x \in \Omega
 .
\end{align}
% pointwise as an $\Lambda^{k}\bbR^{\ast}_n$-valued Lebesgue integral.
% \\
% 
We first observe that $R^{k}_{\eps{\mathtt h}}$ is a bounded linear operator from 
$L^{p}\Lambda^{k}(\Omega^{e})$ into $C\Lambda^{k}(\Omega)$. 
For the mollification operator to yield continuous differential forms,
it is crucial that $\Phi_{\eps{\mathtt h}}$ has continuous first derivatives.

\begin{lemma}
 \label{prop:mollifier}
 Let $\eps > 0$ be small enough.
 The operator 
 \begin{align*}
  R^{k}_{\eps {\mathtt h}} : L^{p}\Lambda^{k}(\Omega^{e}) \rightarrow C\Lambda^{k}(\Omega),
  \quad 
  p \in [1,\infty],%
 \end{align*}
 is well-defined and linear, and we have 
%  and for every $p \in [1,\infty]$ we have
 \begin{align}
  \label{math:mollifier_estimate_local}
  \| R^{k}_{\eps {\mathtt h}} \omega \|_{C\Lambda^{k}(T)}
  &\leq
%   2^{k}
  \left( 1 + \eps L_h \right)^{k}
  \vol^{n}(B_1(0))
  C_{h}^{ \frac{n}{p} }
  \eps^{ - \frac{n}{p} }
  h_T^{ - \frac{n}{p} }
  \| \omega \|_{L^p\Lambda^{k}\left( B_{C_{h} \eps h_T}(T) \right) }
 \end{align}
 for every $p \in [1,\infty]$, $T \in \calT^{n}$ and $\omega \in L^{p}\Lambda^{k}(\Omega^{e})$.
 Moreover, for $\omega \in L^{p,q}\Lambda^{k}(\Omega^{e})$
 with $p,q \in [1,\infty]$ we have 
 \begin{align*}
  R^{k+1}_{\eps {\mathtt h}} \cartan \omega \in C\Lambda^{k+1}(\Omega),
  \quad
  \cartan R^{k}_{\eps {\mathtt h}} \omega = R^{k+1}_{\eps {\mathtt h}} \cartan \omega
  .
 \end{align*}
 \textit{Constants: it suffices that $C_{h} \eps < \eps_h$.}
%  \textit{Constants: it suffices that $C_{h} \eps < \eps_h$.
%  We may choose $C_R = \left( 1 + \eps L_h \right)^{k} \vol^{n}(B_1(0)) C_{h}^{ \frac{n}{p} }$.}
\end{lemma}

\begin{proof}
%  Let $\eps < C_h^{\inv} \eps_h$ and $L_h \eps < \niceonehalf$,
 Let $p \in [1,\infty]$ and let $\omega \in L^{p}\Lambda^{k}(\Omega^{e})$.
 If $\eps < \eps_h C_h^{\inv}$ and $\eps < \niceonehalf L_h$,
 then $\Phi_{\eps \mathtt h, y}$ is a LIP embedding from $\Omega$ to $\Omega^{e}$
 for every $y \in B_1(0)$. 
 Hence $\mu(y) (\Phi_{\eps {\mathtt h},y}^{\ast} \omega)_{|x}$
 is measurable in $y$ for every $x \in \overline\Omega$,
 and the integral \eqref{math:mollification_operator} is well-defined. 
 
 We prove the estimate \eqref{math:mollifier_estimate_local} pointwise.
 Let 
%  $\omega$ be written as in \eqref{math:standardrepresentation_fields}, and let
 $x \in T$ for some $T \in \calT^{n}$.
 Similar as in the proof of Lemma~\ref{prop:pullbackestimate} we observe
 \begin{align*}
  \left| R^{k}_{\eps {\mathtt h}}\omega \right|_{|x}
  \leq 
  \int_{\bbR^{n}} 
  \mu( y ) 
  \left\| \Jacobian_{x} \Phi_{\eps \mathtt h,y} \right\|^{k}_{|x} 
  \|\omega\|_{\Phi_{\eps \mathtt h,y}(x)}
  .
 \end{align*}
 Since integrand vanishes for $y \notin B_1(0)$,
 we may use that 
 \begin{align*}
  \left\| \Jacobian_{x} \Phi_{\eps \mathtt h,y} \right\|^{k}_{|x} 
  \leq 
  \Lip( \Phi_{\eps \mathtt h,y}, \overline\Omega )
  \leq 
  \left( 1 + \eps L_h \right)
 \end{align*}
 and, via a substitution of variables and H\"older's inequality, use 
 \begin{align*}
  \int_{\bbR^n}
  |\mu( y )|
  \cdot 
  \|\omega\|_{ x + \eps {\mathtt h}(x) y }
  \dif y
%   &=
%   \int_{\bbR^n} 
%   \eps^{-n} {\mathtt h}(x)^{-n} \mu( \eps^{\inv} {\mathtt h}(x)^\inv y )
%   \big|{\omega_\sigma\left( x + y \right)}\big|
%   \dif y
%   \\&
  \leq \vol^{n}(B_1(0)) \cdot 
  \eps^{ - \frac{n}{p} }
  {\mathtt h}(x)^{ - \frac{n}{p} }
  \| \omega \|_{L^p\left( B_{\eps {\mathtt h}(x)}(x) \right) }
  .
 \end{align*}
 Both estimates in combination deliver \eqref{math:mollifier_estimate_local}.
 
 Now we show that $R^{k}_{\eps {\mathtt h}}\omega$ is continuous over $\overline\Omega$.
 Let $\omega$ be written as in \eqref{math:standardrepresentation_fields}, 
 and $x \in \overline\Omega$. Then 
 \begin{align*}
  %&\qquad
  R^{k}_{\eps {\mathtt h}} \omega_{|x}
  &=
  \int_{\bbR^{n}} \mu( y ) (\Phi_{\eps {\mathtt h},y}^{\ast} \omega)_{|x} \dif y
  \\&=
  \sum_{ \sigma \in \Sigma(k,n) }
  \int_{\bbR^{n}} 
  \mu( y )
  \omega_{\sigma}\left( x + \eps \mathtt h(x) y \right) 
  (\Phi_{\eps {\mathtt h},y}^{\ast} \cartanx^{\sigma} )_{|x}
  \dif y
  \\&=
  \eps^{-n} \mathtt h(x)^{-n}
  \sum_{ \sigma \in \Sigma(k,n) }
  \int_{\bbR^{n}} 
  \mu\left( \eps^{\inv} \mathtt h(x)^{\inv}( y - x ) \right)
  \omega_{\sigma}( y ) \cdot  W^{\sigma}_{x,y}
  \dif y
  ,%
 \end{align*}
 where we have written 
 \begin{align*}
  W^{\sigma}_{x,y}
  :=
  (\Phi_{\eps {\mathtt h}, \eps^{\inv} \mathtt h(x)^{\inv}( y - x )}^{\ast} \cartanx^{\sigma} )_{|x}
  .
 \end{align*}
 We recall that $\mathtt h$ and $\Phi$ are smooth,
 that $\omega_{\sigma} \in L^{1}(\Omega)$, and that $\overline\Omega$ is compact. 
 In particular, the derivatives of $\Phi$ are continuous. 
 The desired continuity is now a simple consequence 
 of the dominated convergence theorem. 
 
 It remains to show the commutativity property.
 Let $\eta \in C_c^{\infty}\Lambda^{n-k-1}(\Omega)$. 
 By Fubini's theorem we have 
 \begin{align*}
  \int_{\Omega} R^{k}_{\eps {\mathtt h}} \omega \wedge \cartan\eta
  =
  \int_{\Omega} 
  \int_{\bbR^n}
  \mu(y) \Phi_{\eps {\mathtt h},y}^{\ast} \omega \;\dif y \wedge \cartan\eta
  =
  \int_{\bbR^n}
  \mu(y) 
  \int_{\Omega} 
  \Phi_{\eps {\mathtt h},y}^{\ast} \omega \wedge \cartan\eta
  \;\dif y
  ,
  \\
  \int_{\Omega} R^{k}_{\eps {\mathtt h}} \cartan\omega \wedge \eta
  = 
  \int_{\Omega} 
  \int_{\bbR^n}
  \mu(y) 
  \Phi_{\eps {\mathtt h},y}^{\ast} \cartan \omega
  \;\dif y
  \wedge \eta
  = 
  \int_{\bbR^n}
  \mu(y) 
  \int_{\Omega} 
  \Phi_{\eps {\mathtt h},y}^{\ast} \cartan \omega \wedge \eta
  \;\dif y
  .
 \end{align*}
 When $\eps > 0$ is small enough,
 then $\Phi_{\eps {\mathtt h},y} : \Omega \rightarrow \Omega^{e}$
 is a LIP embedding for every $y \in B_1(0)$.
 Hence by Lemma~\ref{prop:pullback:Lpq} we find 
 \begin{align*}
  \int_{\Omega} 
  \Phi_{\eps {\mathtt h},y}^{\ast} \omega \wedge \cartan\eta
  =
  (-1)^{k+1}
  \int_{\Omega} 
  \cartan \Phi_{\eps {\mathtt h},y}^{\ast} \omega \wedge \eta
  =
  (-1)^{k+1}
  \int_{\Omega} 
  \Phi_{\eps {\mathtt h},y}^{\ast} \cartan \omega \wedge \eta
  .
 \end{align*}
 By definition, $\cartan R^{k}_{\eps {\mathtt h}} \omega = R^{k+1}_{\eps {\mathtt h}} \cartan \omega$.
 The proof is complete. 
\end{proof}

\begin{remark}
 \label{rem:mollifier}
 Our Lemma~\ref{prop:mollifier} is analogous to prior findings in literature.
 Let us briefly review the situation.
 The smoothed projection constructed in \cite{AFW1} 
 applies to \emph{quasi-uniform} families of triangulations.
 A family of triangulations is called quasi-uniform
 if for each triangulation $\calT$ in that family we have 
 \begin{gather}
  \forall T \in \calT^{n}
  : h_T^n \leq C_{mesh} |T|,
  \\
  \forall T, S \in \calT 
  : 
  h_T \leq C_{mesh} h_S,
 \end{gather}
 with a common constant $C_{mesh} > 0$.
 In that case, a classical mollification operator can be used 
 instead of our $R^{k}_{\eps\mathtt h}$.
 That result was expanded in \cite{christiansen2008smoothed} 
 to include \emph{shape-uniform} families of triangulations,
 which means that the conditions \eqref{math:shapeconditions:flatness}
 and \eqref{math:shapeconditions:comparison} are satisfied 
 for all triangulations $\calT$ in that family with a common constant $C_{mesh}$.
 The Lipschitz continuous mesh size function of Lemma~\ref{prop:meshsizefunction} 
 was introduced first in \cite{christiansen2008smoothed}.
 But simple examples show that, contrarily to the statement in \cite{christiansen2008smoothed},
 a regularization operator with that mesh size function does not yield 
 a continuous differential form.
 This is due to the differential of the mesh size function being discontinuous in general.
 As a remedy, we explicitly construct a mesh size function that is smooth.
%  This idea has also been used in \cite{StructPresDisc}.
 
 The Lipschitz continuous mesh size function in Lemma~\ref{prop:meshsizefunction}
 is the limit of the smoothed mesh size function in Lemma~\ref{prop:smoothmeshsizefunction}
 for decreasing mollification radius. It is natural to ask how this limit process
 is reflected in the regularization operator. 
 It is easily seen that the gradient of the original mesh size function
 features tangential continuity. Using this additional property,
 one can show that the regularization operator of \cite{christiansen2008smoothed}
 does yield differential forms that are piecewise continuous with respect
 to the triangulation and that are single-valued along simplex boundaries.
 Consequently, the regularized differential form, though not continuous, 
 still has well-defined degrees of freedom,
 and the finite element interpolator can be applied as intended.
 We emphasize that the main result of \cite{christiansen2008smoothed} remains unchanged.
\end{remark}

\subsection{Smoothed Interpolation and Smoothed Projection}

Combining the extension operator, the mollification operator, 
and the finite element interpolator,
we provide the \emph{smoothed interpolator}
\begin{align}
 Q^{k}_{\eps} 
 :
 L^{p}\Lambda^{k}(\Omega)
 \rightarrow
 L^{p}\Lambda^{k}(\calT),
 \quad 
 \omega
 \mapsto
 I^{k} R^{k}_{\eps {\mathtt h}} E^{k} \omega,
 \quad 
 p \in [1,\infty]. 
\end{align}
We show that $Q^{k}_{\eps}$ satisfies local bounds 
and commutes with the exterior derivative:

\begin{theorem}
 \label{prop:interpolationbound}
 Let $\eps > 0$ be small enough.
 For $p \in [1,\infty]$,
 the operator $Q^{k}_{\eps} : L^{p}\Lambda^{k}(\Omega) \rightarrow L^{p}\Lambda^{k}(\calT)$ is linear and bounded,
 and there exists uniformly bounded $C_{Q,p} > 0$ such that 
 \begin{align}
  \label{prop:interpolationbound:localestimate}
  \| Q^{k}_{\eps} \omega \|_{L^{p}\Lambda^{k}(T)}
  &\leq
  C_{Q,p} \eps^{ -\frac{n}{p} } \| \omega \|_{L^{p}\Lambda^{k}(\calT(T))},
  \quad 
  \omega \in L^{p}\Lambda^{k}(\Omega),
  \quad T \in \calT^{n}, 
  \\
  \label{prop:interpolationbound:globalestimate}
  \| Q^{k}_{\eps} \omega \|_{L^{p}\Lambda^{k}(\Omega)}
  &\leq
  C_N^{\frac{1}{p}} C_{Q,p} \eps^{ -\frac{n}{p} } \| \omega \|_{L^{p}\Lambda^{k}(\Omega)},
  \quad 
  \omega \in L^{p}\Lambda^{k}(\Omega)
  .
 \end{align}
 Moreover, we have 
 \begin{align}
  \label{prop:interpolationbound:commutativity}
  \cartan Q^{k}_{\eps} \omega = Q^{k}_{\eps} \cartan \omega,
  \quad 
  \omega \in L^{p,q}\Lambda^{k}(\Omega),
  \quad 
  p, q \in [1,\infty].
 \end{align}
 \textit{Constants:
 it suffices that $\eps > 0$ is small enough to apply Lemma~\ref{prop:mollifier}
 and that $L_{\Psi} C_h \eps < \eps_h$.
 We may choose 
 $C_{Q,p}
 = 
%  2^{k}
 (1 + \eps L_h)^{k}
 \vol^{n}(B_1(0)) 
%  \binom{n}{k} %FIXME: Modify proof to get rid of this here.
%  C_{mesh}^{ \frac{n}{p} } 
 C_{h}^{ \frac{n}{p} } c_{M}^{k} C_{M}^{k} C_I C_{E,p}$.
 }
\end{theorem}

\begin{proof}
 Let $\omega \in L^{p}\Lambda^{k}(\Omega)$ and let $T \in \calT^{n}$.
 Then 
 \begin{align*}
  \| Q^{k}_{\eps} \omega \|_{L^{p}\Lambda^{k}(T)}
  &\leq 
  \| I^{k} R^{k}_{\eps {\mathtt h}} E^{k} \omega \|_{L^{p}\Lambda^{k}(T)}
  \\&\leq 
  |T|^{\frac{1}{p}} \| I^{k} R^{k}_{\eps {\mathtt h}} E^{k} \omega \|_{L^{\infty}\Lambda^{k}(T)}
  \leq 
%   C_{mesh}^{n/p} 
  h_T^{\frac{n}{p}}\| I^{k} R^{k}_{\eps {\mathtt h}} E^{k} \omega \|_{L^{\infty}\Lambda^{k}(T)}
  . 
 \end{align*}
 Estimate \eqref{math:feminterpolator:simplexbound} gives 
 \begin{align*}
  \| I^{k} R^{k}_{\eps {\mathtt h}} E^{k} \omega \|_{L^{\infty}\Lambda^{k}(T)}
  &\leq 
  C_M^{k} h_T^{-k} 
  \| A_T^{\ast} I^{k} R^{k}_{\eps {\mathtt h}} E^{k} \omega \|_{L^{\infty}\Lambda^{k}(\Delta^{n})}
  \\&\leq 
  C_I C_M^{k} h_T^{-k} 
  \| A_T^{\ast} R^{k}_{\eps {\mathtt h}} E^{k} \omega \|_{L^{\infty}\Lambda^{k}(\Delta^{n})}
  \\&\leq 
  C_{I} c_{M}^{k} C_M^{k} 
  \| R^{k}_{\eps {\mathtt h}} E^{k} \omega \|_{C\Lambda^{k}(T)}.
 \end{align*}
 Assuming that $\eps > 0$ is small enough, we apply Lemma~\ref{prop:mollifier},
 \begin{align*}
  \| R^{k}_{\eps {\mathtt h}} E^{k} \omega \|_{C\Lambda^{k}(T)}
  &\leq
%   2^{k}
  (1 + \eps L_h)^{k}
  \vol^{n}(B_1(0)) 
%   \binom{n}{k} % FIXME: Modify proof to get rid of this here. 
  \eps^{ -\frac{n}{p} } h_T^{-\frac{n}{p}} C_{h}^{ \frac{n}{p} }
  \| E^{k} \omega \|_{L^p\Lambda^{k}\left( B_{C_{h} \eps h_T}(T) \right) }
  ,  
 \end{align*}
 and find with \eqref{math:neighborcontainment} and Corollary~\ref{prop:extensionestimate:furtherrefined} that 
 \begin{align*}
  \| E^{k} \omega \|_{L^p\Lambda^{k}\left( B_{C_{h} \eps h_T}(T) \right) }
  &=
  \| E^{k} \omega \|_{L^p\Lambda^{k}\left( B_{C_{h} \eps h_T}(T) \cap \Omega^{e} \right) }
  \\&\leq 
  C_{E,p}
  \| \omega \|_{L^p\Lambda^{k}\left( B_{L_{\Psi} C_{h} \eps h_T}(T) \cap \overline\Omega \right) }
  \\&\leq 
  C_{E,p}
  \| \omega \|_{L^p\Lambda^{k}\left( B_{ \eps_h h_T}(T) \cap \overline\Omega \right) }
  \leq 
  C_{E,p}
  \| \omega \|_{L^p\Lambda^{k}\left( \calT(T) \right) }
  . 
 \end{align*}
 Thus the local bound \eqref{prop:interpolationbound:localestimate} follows.
 The global bound \eqref{prop:interpolationbound:globalestimate}
 is obtained via 
 \begin{align*}
  \| Q^{k}_{\eps} \omega \|_{L^{p}\Lambda^{k}( \Omega )}^{p}
  &=
  \sum_{ T \in \calT^{n} }
  \| Q^{k}_{\eps} \omega \|_{L^{p}\Lambda^{k}( T )}^{p}
  \leq C_{Q,p}^{p}  
  \sum_{ T \in \calT^{n} }
  \| \omega \|_{L^{p}\Lambda^{k}( \calT(T) )}^{p}
  \\&\leq C_{Q,p}^{p} C_{N}  
  \sum_{ T \in \calT^{n} }
  \| \omega \|_{L^{p}\Lambda^{k}( T )}^{p}
  \leq C_{Q,p}^{p} C_{N}  
  \| \omega \|_{L^{p}\Lambda^{k}( \Omega )}^{p}
 \end{align*}
 for $p \in [1,\infty)$, and for $p = \infty$ similarly. 
 Finally, \eqref{prop:interpolationbound:commutativity}
 follows from Theorem~\ref{prop:extensionderivativecommute},
 Theorem~\ref{prop:mollifier}, and our assumptions on $I^{k}$. 
 The proof is complete. 
\end{proof}

%%%%%%%%%%%%%%%%%%%%%%%%%%%%%%%%%%
%%%%%%%%%%%%%%%%%%%%%%%%%%%%%%%%%%
%%%%%%%%%%%%%%%%%%%%%%%%%%%%%%%%%%
%%%%%%%%%%%%%%%%%%%%%%%%%%%%%%%%%%
%%%%%%%%%%%%%%%%%%%%%%%%%%%%%%%%%%
%%%%%%%%%%%%%%%%%%%%%%%%%%%%%%%%%%
%%%%%%%%%%%%%%%%%%%%%%%%%%%%%%%%%%
%%%%%%%%%%%%%%%%%%%%%%%%%%%%%%%%%%
%%%%%%%%%%%%%%%%%%%%%%%%%%%%%%%%%%
%%%%%%%%%%%%%%%%%%%%%%%%%%%%%%%%%%
%%%%%%%%%%%%%%%%%%%%%%%%%%%%%%%%%%
%%%%%%%%%%%%%%%%%%%%%%%%%%%%%%%%%%

The smoothed interpolator $Q_{\eps}^{k}$ is local and satisfies uniform bounds.
Although $Q_{\eps}^{k}$ generally does not reduce to the identity over $\Lambda^{k}_{}(\calT)$,
we can show that, for $\eps > 0$ small enough, it is close to the identity
and satisfies a local error estimate.

\begin{theorem}
 \label{prop:interpolation_error}
 For $\eps > 0$ small enough, 
 there exists uniformly bounded $C_{e,p} > 0$ for every $p \in [1,\infty]$ 
 such that 
 \begin{align*}
  \| \omega - Q_{\eps}^{k} \omega \|_{L^{p}\Lambda^{k}(T)}
  \leq 
  \epsilon C_{e,p} 
  \| \omega \|_{L^{p}\Lambda^{k}(\calT(T))},
  \quad 
  \omega \in \Lambda^{k}_{}(\calT),
  \quad 
  T \in \calT^{n}
  .
 \end{align*}
 % C_I C_{mesh}^{n/p} C_{M}^{k} 
 % k=0: C_{E,\infty} C_h C_M \epsilon C_{\flat,p,k}
 % k>0: \epsilon C_h C_M l^{k} ( 1 + C_{\partial} ) \cdot C_{E,\infty} C_{M}^{k+1} C_{\flat,p,k} 
 % 
 % k=0: \epsilon C_h C_M C_{E,\infty} C_{\flat,p,k}
 % k>0: \epsilon C_h C_M C_{E,\infty} C_{\flat,p,k} l^{k} ( 1 + C_{\partial} ) C_{M}^{k+1} 
 % 
 % TOTAL:
 % C_I C_{mesh}^{n/p} C_{M}^{k} \epsilon C_h C_M C_{E,\infty} C_{\flat,p,k} l^{k} ( 1 + C_{\partial} ) C_{M}^{k+1}
 % C_I C_{mesh}^{n/p} C_{M}^{2k+2} \epsilon C_h C_{E,\infty} C_{\flat,p,k} l^{k} ( 1 + C_{\partial} )
 % 
 \textit{Constants:
 It suffices that $\eps > 0$ is small enough such that Theorem~\ref{prop:interpolationbound} is applicable
 and $L_{\Psi} C_{M}C_{h} \epsilon < \eps_h$.
 We may choose
 \begin{align*}
  C_{e,p}
  =
%   C_I C_{mesh}^{n/p} C_{M}^{2k+2} C_h C_{E,\infty} C_{\flat,p,k} \left( 1 + \eps C_M L_h \right)^{k} ( 1 + C_{\partial} )
%   .
  c_M^{2k+1} C_{M}^{2k+2+\frac{n}{p}} 
  C_I C_h 
  \left( 1 + c_M C_M L_h \epsilon \right)^{k}
  ( 1 + C_{\partial} ) 
  C_{E,\infty} C_{\flat,p,k} 
  . 
 \end{align*}
% 
%  
%  with $l = ( 1 + \eps C_{M} L_{h} )$.
 }
\end{theorem}

\begin{proof}
 We prove the statement by a series of inequalities. 
 Let $\omega \in \Lambda_{}^{k}(\calT)$ and let $T \in \calT^{n}$. 
 Then 
 \begin{align*}
  \|  \omega - Q_{\eps}^{k} \omega \|_{L^{p}\Lambda^{k}(T)}
  &=
  \| I^{k} ( E^{k} \omega - R^{k}_{\eps \mathtt h} E^{k} \omega ) \|_{L^{p}\Lambda^{k}(T)}
%   \\&\leq 
%   |T|^{\frac{n}{p}}
%   \| I^{k} ( E^{k} \omega - R^{k}_{\eps \mathtt h} E^{k} \omega ) \|_{L^{\infty}\Lambda^{k}(T)}
  \\&\leq 
%   C_{mesh}^{\frac{n}{p}} 
  h_T^{\frac{n}{p}}
  \| I^{k} ( E^{k} \omega - R^{k}_{\eps \mathtt h} E^{k} \omega ) \|_{L^{\infty}\Lambda^{k}(T)}
  \\&\leq 
%   C_{mesh}^{\frac{n}{p}} 
  C_{M}^{k} h_T^{\frac{n}{p}-k}
  \| A_T^{\ast} I^{k} ( E^{k} \omega - R^{k}_{\eps \mathtt h} E^{k} \omega ) \|_{L^{\infty}\Lambda^{k}(A_T^{\inv} T)}
  ,
 \end{align*}
 as follows from \eqref{math:referencetransformation}.
 By \eqref{math:feminterpolator:micro} and \eqref{math:pushforward_pullback_duality}, we have 
 \begin{align*}
  \| A_T^{\ast} I^{k} ( E^{k} \omega - R^{k}_{\eps \mathtt h} E^{k} \omega ) \|
  _{L^{\infty}\Lambda^{k}(A_T^{\inv} T)}
  \leq 
  C_I
  \sup_{ \substack{ F \subseteq T \\ S \in \calC_{k}^{F} } }
  |A_{T\ast}^{\inv}S|^{\inv}_{k}
%   \int_{A_{T\ast}^{\inv}S} 
%    A_T^{\ast} E^{k} \omega
%    -
%    A_T^{\ast} R^{k}_{\eps \mathtt h} E^{k} \omega
   \int_{S} 
   E^{k} \omega
   -
   R^{k}_{\eps \mathtt h} E^{k} \omega
   .
 \end{align*}
 We need to bound the last expression.
 Fix $F \in \Delta(T)$ and $S \in \calC^{F}_{k}$.
 We see that 
 \begin{align*}
  \int_{S}
   E^{k} \omega
   - R^{k}_{\eps \mathtt h} E^{k} \omega
  =
  \int_{S} 
  \int_{\bbR^{n}} \mu(y) 
  \left( 
   E^{k} \omega
   -
   \Phi_{\eps {\mathtt h},y}^{\ast} E^{k} \omega \right)
  \dif y.
 \end{align*}
 By assumption on $S$,
 both integrals are taken in the sense of measure theory,
 and we may apply Fubini's theorem:
 \begin{align*}
  \int_{S} 
  \int_{\bbR^{n}} \mu(y) 
   \Phi_{\eps {\mathtt h},y}^{\ast} E^{k} \omega
  \;\dif y
  =
  \int_{\bbR^{n}} \mu(y) 
  \int_{S} 
   \Phi_{\eps {\mathtt h},y}^{\ast} E^{k} \omega  
  \;\dif y.
 \end{align*}
 Using these observations and \eqref{math:pushforward_pullback_duality} again, we have 
 \begin{align*}
  \int_{\bbR^{n}} \mu(y) 
  \int_{ S} 
   E^{k} \omega - \Phi_{\eps {\mathtt h},y}^{\ast} E^{k} \omega
  \dif y
  =
  \int_{\bbR^{n}} \mu(y) 
  \int_{ A_{T\ast}^{\inv} S - A_{T\ast}^{\inv} \Phi_{\eps {\mathtt h},y \ast} S } 
   A_{T}^{\ast} E^{k} \omega
  \;\dif y
  .
 \end{align*}
 With \eqref{math:dualityestimate:flat}, it follows that 
 \begin{align*}
  &
  \int_{\bbR^{n}} \mu(y) 
  \int_{ A_{T\ast}^{\inv} S - A_{T\ast}^{\inv} \Phi_{\eps {\mathtt h},y\ast} S } A_T^{\ast} E^{k} \omega 
  \; \dif s \dif y
  \\&\quad\quad\leq
  \sup_{ y \in B_1(0) }
  \| A_{T\ast}^{\inv} S - A_{T\ast}^{\inv} \Phi_{\eps {\mathtt h},y\ast} S \|_{k,\flat}
  \cdot 
  \| A_T^{\ast} E^{k} \omega \|_{L^{\infty,\infty}\Lambda^{k}(B_{C_{M}C_{h}\epsilon}(\Delta^{n}))}  
  .
 \end{align*}
 We need to bound this product. 
 On the one hand, we observe that 
 \begin{gather*}
  \sup_{ \substack{ x \in A_T^{-1}(F) \\ y \in B_1(0) } }
%   \sup_{ \substack{ y \in B_1(0) } }
  | x - A_T^{-1} \Phi_{\eps \mathtt h,y} A_T^{}(x) |
  \leq 
  \sup_{ x \in A_T^{-1}(F) }
  | \epsilon \mathtt h( A_T x ) A_T^{-1} |
  \leq 
  C_{h} C_{M} \epsilon,
%   \quad 
%   x \in F, 
  \\
%  \end{gather*}
%  for $x \in F$, and 
%  \begin{gather*}
  \sup_{ y \in B_1(0) }
  \Lip\left( A_T^{-1} \Phi_{\eps\mathtt h,y} A_T, A_T^{-1}(F) \right)
  \leq 
  1 + c_M C_M L_h \epsilon.
 \end{gather*}
 By Lemma~\ref{prop:deformationestimate}, we then estimate
 \begin{align*}
  &\sup_{ y \in B_1(0) }
  \| A_{T\ast}^{\inv} S - A_{T\ast}^{\inv} \Phi_{\eps {\mathtt h},y\ast} S \|_{k,\flat}
  = 
  \sup_{ y \in B_1(0) }
  \| A_{T\ast}^{\inv} S - A_{T\ast}^{\inv} \Phi_{\eps {\mathtt h},y\ast} A_{T\ast} A_{T\ast}^{\inv} S \|_{k,\flat}
  \\&\qquad\leq 
  \eps C_{h} C_{M} \left( 1 + c_M C_M L_h \epsilon \right)^{k}
  \left( |A_{T\ast}^{\inv} S|_{k} + |\partial A_{T\ast}^{\inv} S|_{k-1} \right)
  .
 \end{align*}
%  where $l := \sup\{ 1, 1 + \eps C_M L_h \}$,
%  where $l := 1 + \eps C_M L_h$.
 The inverse inequality \eqref{math:inverseinequality:dof} gives 
 \begin{align*}
  | \partial A_{T\ast}^{\inv} S |_{k-1}  
  \leq 
  C_{\partial} 
  | A_{T\ast}^{\inv} S |_{k}  
  .
 \end{align*}
 On the other hand, we observe 
 \begin{align*}
  \| A_T^{\ast} E^{k} \omega \|_{L^{\infty,\infty}\Lambda^{k}(B_{C_{M}C_{h}\epsilon}(\Delta^{n}))}  
  \leq
  C_{E,\infty} c_M^{k+1} C_M^{k+1}
  \| A_T^{\ast} \omega \|_{L^{\infty,\infty}\Lambda^{k}(A_T^{\inv}\calT(T))}.  
 \end{align*}
 To see this, we let $\epsilon > 0$ be small enough
 and apply Corollary~\ref{prop:extensionestimate:furtherrefined}
 to obtain 
 \begin{align*}
%   & 
  \| A_T^{\ast} E^{k} \omega \|_{L^{\infty}\Lambda^{k}(B_{C_{M}C_{h}\epsilon}(\Delta^{n}))}  
  &\leq 
  c_M^{k} h_T^{k}
  \| E^{k} \omega \|_{L^{\infty}\Lambda^{k}(B_{C_{M}C_{h} \epsilon h_T}(T))}  
  \\&\leq 
  C_{E,\infty} c_M^{k} h_T^{k}
  \| \omega \|_{L^{\infty}\Lambda^{k}( B_{L_{\Psi} C_{M}C_{h} \epsilon h_T}(T) \cap \overline\Omega)}  
  \\&\leq 
  C_{E,\infty} c_M^{k} h_T^{k}
  \| \omega \|_{L^{\infty}\Lambda^{k}(\calT(T))}  
  \\&\leq 
  C_{E,\infty} c_M^{k} C_M^{k} 
  \| A_T^{\ast} \omega \|_{L^{\infty}\Lambda^{k}(A_T^{\inv}\calT(T))}  
  .
 \end{align*}
 We treat $A_T^{\ast} E^{k+1} \cartan\omega$ similarly.
 The inverse inequality \eqref{math:inverseinequality:fe} gives 
 \begin{align*}
  \| A_T^{\ast} \omega \|_{L^{\infty,\infty}\Lambda^{k}(A_T^{\inv}\calT(T))}
  \leq
  C_{\flat,p,k}
  \| A_T^{\ast} \omega \|_{L^{p}\Lambda^{k}(A_T^{\inv}\calT(T))}
  .
 \end{align*}
 In combination, it follows that 
 \begin{align*}
  &
  \| A_T^{\ast} I^{k} ( \omega - R^{k}_{\eps \mathtt h} E^{k} \omega ) \|
  _{L^{\infty}\Lambda^{k}(\Delta^{n})}
  \\&\quad \leq 
  C_I C_h C_M \left( 1 + c_M C_M L_h \epsilon \right)^{k} ( 1 + C_{\partial} ) 
  C_{E,\infty} c_M^{k+1} C_M^{k+1} C_{\flat,p,k} 
  \epsilon
  \| A_T^{\ast} \omega \|_{L^{p}\Lambda^{k}(\calT(T))}
  .
 \end{align*}
 We finally recall that 
 \begin{align*}
  \| A_T^{\ast} \omega \|_{L^{p}\Lambda^{k}(A_T^{\inv}\calT(T))}
  \leq 
  c_M^{k} C_{M}^{\frac{n}{p}} h_T^{k-\frac{n}{p}} 
  \| \omega \|_{L^{p}\Lambda^{k}(\calT(T))}
  .
 \end{align*}
 This completes the proof.
\end{proof}

\begin{remark}
 \label{rem:interpolationerror}
 Our Theorem~\ref{prop:interpolation_error} resembles Lemma 5.5 in \cite{AFW1}
 and Lemma 4.2 in \cite{christiansen2008smoothed}.
 Let us briefly motivate why we use a different method of proof. 
 In order to obtain the interpolation error estimate over simplices $T \in \calT$,
%  that are contained in the interior of $\Omega$,
 the authors of the aforementioned references 
 suppose that finite element differential forms are piecewise Lipschitz near $T$.
 This holds true if $T$ is an interior simplex
 but not if $T$ touches the boundary of $\Omega$,
 and it is not clear how their method applies for such $T$.
 The reason is that their extension operator, like ours, involves a pullback along a bi-Lipschitz mapping,
 so the extended finite element differential form 
 is not necessarily Lipschitz continuous anywhere outside of $\Omega$. 
 The extended differential form, however, is still a flat form,
 and this motivates our utilization of geometric measure theory 
 to prove the desired estimate for the interpolation error. 
 
 For strongly Lipschitz domains, Lipschitz collars with stronger regularity
 may provide an alternative remedy,
 but we do explore this idea further in this article. 
%  This is not explored in this article. 
%  Alternatively, one might show that strongly Lipschitz domains
%  admit a Lipschitz collar with additional regularity. 
%  
%  
%  where finite element differential forms are extended by reflection along a normal field.
%  
%  The proofs in those references utilize piecewise Lipschitz regularity 
%  of the finite element differential form.
%  For interior simplices, the differential form is indeed piecewise Lipschitz.
%  But along boundary simplices, the mollified differential form takes into account
%  values of its extension. Since the extension operator is assumed to be merely bi-Lipschitz,
%  the extended differential form is generally not Lipschitz along the boundary,
%  and different techniques need to be applied there.
%  The extended differential form is a flat form, however, and this motivates our utilization
%  to geometric measure theory to obtain the desired estimate for the interpolation error. 
\end{remark}

We are now in a position to prove the main result of this article.
For $\eps > 0$ small enough, we can correct the error of the smoothed interpolation
over the finite element space.
The resulting smoothed projection is, however, non-local.

\begin{theorem}
 \label{prop:finalprojection}
 Let $\eps > 0$ be small enough.
 There exists a bounded linear operator 
 \begin{align*}
  \pi^{k}_{\eps} : L^{p}\Lambda^{k}(\Omega) \rightarrow L^{p}\Lambda^{k}(\calT),
  \quad 
  p \in [1,\infty],
 \end{align*}
 such that 
 \begin{align*}
  \pi^{k}_{\eps} \omega = \omega, \quad \omega \in \Lambda^{k}_{}(\calT),
 \end{align*}
 such that  
 \begin{align*}
  \cartan \pi^{k}_{\eps} \omega = \pi^{k}_{\eps} \cartan \omega,
  \quad 
  \omega \in L^{p,q}\Lambda^{k}(\Omega),
  \quad
  p,q \in [1,\infty],
 \end{align*}
 and such that for all $p \in [1,\infty]$ there exist uniformly bounded $C_{\pi,p} > 0$ with
 \begin{align*}
  \| \pi^{k}_{\eps} \omega \|_{L^{p}\Lambda^{k}(\calT)}
  \leq 
  C_{\pi,p} \eps^{ -\frac{n}{p} } \| \omega \|_{L^{p}\Lambda^{k}(\Omega)},
  \quad 
  \omega \in L^{p}\Lambda^{k}(\Omega).
 \end{align*}
 \textit{Constants: it suffices that $\eps > 0$ is so small 
 that Theorem~\ref{prop:interpolationbound} and Theorem~\ref{prop:interpolation_error} apply,
 and that $C_{e,p} \eps < 2$.
 We may assume $C_{\pi,p} = 2 C_{Q,p} C_{N}^{\frac{1}{p}}$.
 } 
\end{theorem}

\begin{proof}
 If $\eps > 0$ is small enough and $p \in [1,\infty]$, 
 then Theorem~\ref{prop:interpolation_error} implies that  
 \begin{align*}
  \| \omega - Q^{k}_{\eps} \omega \|_{L^{p}\Lambda^{k}(\Omega)}
  \leq \frac{1}{2}
  \| \omega \|_{L^{p}\Lambda^{k}(\Omega)}
  ,
  \quad
  \omega \in \Lambda^{k}_{}(\calT)
  . 
 \end{align*}
 By standard results, the operator
 $Q^{k}_{\eps} : L^{p}\Lambda^{k}(\calT) \rightarrow L^{p}\Lambda^{k}(\calT)$ is invertible.
 Let $J^{k}_{\eps} : L^{p}\Lambda^{k}(\calT) \rightarrow L^{p}\Lambda^{k}(\calT)$ be its inverse.
 $J^{k}_{\eps}$ does not depend on $p$, since $Q^{k}_{\eps}$ does not depend on $p$.
%  Since $Q^{k}_{\eps}$ is independent of $p$,
%  so is $J^{k}_{\eps}$.
 The construction of $J^{k}_{\eps}$ via a Neumann series reveals that 
% $J^{k}_{\eps}$ is bounded: 
 \begin{align*}
  \| J^{k}_{\eps} \omega \|_{L^{p}\Lambda^{k}(\Omega)} \leq 2 \| \omega \|_{L^{p}\Lambda^{k}(\Omega)},
  \quad
  \omega \in \Lambda^{k}_{}(\calT).
 \end{align*}
 So $J^{k}_{\eps}$ is bounded. 
 Moreover, $J^{k}_{\eps}$ commutes with the exterior derivative because 
 \begin{align*}
  \cartan J^{k}_{\eps} \omega
  =
  J^{k}_{\eps} Q^{k}_{\eps} \cartan J^{k}_{\eps} \omega
  =
  J^{k}_{\eps} \cartan Q^{k}_{\eps} J^{k}_{\eps} \omega
  =
  J^{k}_{\eps} \cartan \omega
  ,
  \quad 
  \omega \in \Lambda^{k}_{}(\calT).
 \end{align*}
%  it commutes with the exterior derivative.
 The theorem follows with $\pi^{k}_{\eps} := J^{k}_{\eps} Q^{k}_{\eps}$.
%  We conclude that 
%  \begin{align*}
%   \pi^{k}_{\eps} = J^{k}_{\eps} Q^{k}_{\eps}
%  \end{align*}
%  satisfies the desired properties.
%  The proof is complete.
\end{proof}

\begin{remark}
 Several quantities in this section depend on a Lebesgue exponent $p \in [1,\infty]$,
 but it suffices to consider only the case $p = 1$.
 We carefully observe that  
 \begin{align*}
  C_{\calA,p} \leq C_{\calA,1},
  \quad 
  C_{E,p} \leq C_{E,1},
  \quad 
  C_{Q,p} \leq C_{Q,1},
  \quad 
  C_{e,p} \leq C_{e,1},
  \quad 
  C_{\pi,p} \leq C_{\pi,1},
 \end{align*}
 for all $p \in [1,\infty]$.
%  In particular, $C_{e,p} \eps < 2$ follows from $C_{e,1} \eps < 2$.
 Hence a sufficiently small choice of $\eps > 0$ is sufficient 
 to enable Theorem~\ref{prop:finalprojection} for all $p \in [1,\infty]$ simultaneously.
\end{remark}

\begin{remark}
 Throughout this section, we have provided explicit formulas
 for the admissible ranges of $\epsilon$ and the various constants.
 With the exception of $C_{\calA,p}$ and $L_{\Omega}$,
 the quantities in those formulas depend only on 
 the ambient dimension, the polynomial degree, and the mesh regularity.
 If bounds for $C_{\calA,p}$ and $L_{\Omega}$ are known,
 then all constants in this section are effectively computable.
%  
%  We have provided expressions for the ranges of $\epsilon$
%  and the various constants in the statements of this section.
%  The parameters in those expressions are computable in principle,
%  with the possibly exception of $C_{\calA,p}$ and $L_{\Omega}$,
%  which depend on the domain and the Lipschitz collar.
%  If estimates for all those parameters are known,
%  then the norm of the smoothed projection is computable.
\end{remark}

\subsection*{Acknowledgments}

The author would like to thank
Douglas N. Arnold and Snorre H. Christiansen for stimulating discussion.
Some of the research of this paper was done while the author was visiting the School of Mathematics at the University of Minnesota, 
whose kind hospitality and financial support is gratefully acknowledged.
This research was supported by the European Research Council through
the FP7-IDEAS-ERC Starting Grant scheme, project 278011 STUCCOFIELDS.

% VII. Bibliography 
\providecommand{\bysame}{\leavevmode\hbox to3em{\hrulefill}\thinspace}
\providecommand{\MR}{\relax\ifhmode\unskip\space\fi MR }
\providecommand{\MRhref}[2]{%
  \href{http://www.ams.org/mathscinet-getitem?mr=#1}{#2}
}
\providecommand{\href}[2]{#2}

\end{document}